\documentclass[10pt]{amsart}
\usepackage{amsmath,mathrsfs,cite}
\usepackage[colorlinks=true,urlcolor=blue,
citecolor=red,linkcolor=blue,linktocpage,pdfpagelabels,bookmarksnumbered,bookmarksopen]{hyperref}
\usepackage{color,mathrsfs}
\newtheorem{definition}{Definition}

\newtheorem{theorem}{Theorem}[section]
\newtheorem{proposition}[theorem]{Proposition}
\newtheorem{lemma}[theorem]{Lemma}

\newtheorem{remark}[theorem]{Remark}

\newcommand{\N}{{\mathbb N}}
\newcommand{\R}{{\mathbb R}}

\def\Om{\Omega}

\def\ra{{\rightarrow}}
\def\na{\nabla}

\def\l{{\lambda}}
\def\a{{\alpha}}

\def\S{{\mathcal{S}}}

\def\D{\Delta}

\def\e{\varepsilon}

\def\de{\partial}
\begin{document}
\title[Symmetry for cooperative elliptic systems]
{Symmetry results for cooperative elliptic systems in unbounded domains}

\author[Damascelli]{Lucio Damascelli}
\address{ Dipartimento di Matematica, Universit\`a  di Roma
" Tor Vergata " - Via della Ricerca Scientifica 1 - 00173  Roma - Italy.}
\email{damascel@mat.uniroma2.it}
\author[Gladiali]{Francesca  Gladiali}
\address{Dipartimento Polcoming-Matematica e Fisica, Universit\`a  di Sassari  - Via Piandanna 4, 07100 Sassari - Italy.}
\email{fgladiali@uniss.it}
\author[Pacella]{Filomena Pacella}
\address{Dipartimento di Matematica, Universit\`a di Roma La Sapienza -  P.le A. Moro 2 - 00185 Roma - Italy.}
\email{pacella@mat.uniroma1.it}
\date{}
\thanks{Supported by PRIN-2009-WRJ3W7 grant}

\begin{abstract}
In this paper we prove symmetry results for classical solutions of semilinear cooperative elliptic systems in $\R^N$, $N\geq 2$ or in the exterior of a ball.
We consider the case of fully coupled systems and nonlinearities which are either convex or have a convex derivative.\\
The solutions are shown to be foliated Schwarz symmetric if a bound on their Morse index holds.
As a consequence of the symmetry results we also obtain some nonexistence theorems.
\end{abstract}
\maketitle

\section{  \textbf{Introduction and statement of the results} }
\label{se:1}
In this paper we study symmetry properties of classical $C^2$ solutions of a semilinear elliptic system of the type
\begin{equation}\label{1}
-\D U=F(|x|,U)\quad  \hbox{ in }\Omega
\end{equation}
where   $\Omega$ is either $\R^N$ or the exterior of a ball (i.e. $\Omega=\R^N\setminus B$ where $B$ is the unit ball centered at the origin), $ N \geq 2$, and
 $F = (f_1 ,\dots, f_m )\,:\,[0, + \infty ) \times\R^m\longrightarrow \R^m$, $m\geq 2$,    is locally a  $C^ {1,\alpha}$ function.
  In the second case   we also require the boundary conditions 
\begin{equation}\label{2}
U=0\quad \hbox{ on }\de \Omega.
\end{equation}
When $\Omega=\R^N$ some radial symmetry results for positive solutions of \eqref{1} have been obtained using the classical moving plane method under further assumptions on $F$ and/or on the decay of the solutions at infinity (\cite{deFY},\cite{BS}).\\
As far as we know there are no results in the case when the solution changes sign or the underlying domain is the exterior of a ball.\\
Here we use the approach introduced in \cite{PA}, \cite{PAW}, \cite{GPW} (see also \cite{PR}) in the scalar case, i.e. when \eqref{1} reduces to a single equation, to prove a partial symmetry result for solutions of \eqref{1} and \eqref{2} having low Morse index, assuming some convexity on the nonlinear term $F(|x|,U)$. This approach is not immediately applicable to the case of systems, as explained in \cite{DAPA}. However in \cite{DAPA} and \cite{DGP} using some new ideas and, in particular, considering an auxiliary symmetric system, foliated Schwarz symmetry of solutions is proved in the case of rotationally invariant bounded domains, i.e. balls or annuli.\\
In passing from bounded to unbounded domains new difficulties arise, some of which are peculiar of the vectorial case and do not appear in the scalar case. Therefore, though our strategy is mostly based on combining the approaches of \cite{GPW} (for the scalar equations in unbounded domains) and  of \cite{DAPA} and \cite{DGP} (for systems in bounded domains), we need some new devices, in particular we derive some other sufficient conditions for foliated Schwarz symmetry (see Section \ref{se:3}).\\
To precisely state our results we need a few definitions.
\begin{definition}\label{d-schwarz-symmetry}
Let $\Omega$ be a rotationally symmetric domain in $\R^N$, $N\geq 2$. We say that a continuous vector valued function $U=(u_1 , \dots ,u_m): \Omega \to \R^m$ is foliated Schwarz symmetric if each component $u_i$ is foliated Schwarz symmetric with respect to the same vector $p \in \R^N$. In other words there exists a vector $p \in \R^N$,
$|p|=1$, such that $U(x)$ depends only on $r=|x|$ and $\theta= \arccos \left ( \frac {x}{|x|}\cdot p  \right ) $ and $U$ is (componentwise) nonincreasing in $\theta $.
\end{definition}
\noindent Next we define the Morse index for a solution $U$ of \eqref{1} and \eqref{2}. To this aim we denote  by $Q_U(-,\Omega)$ the quadratic form corresponding to a solution, i.e.
\begin{equation}\label{forma-quadratica-intr}
Q_U(\psi,\Omega):=\int_{\Omega}\Big[\sum_{i=1}^m|\na \psi_i|^2-\sum_{i,j=1}^m \frac{\de f_i}{\de u_j}(|x|,U)\psi_j\psi_i\Big]\, dx
\end{equation}
where $\psi=(\psi_1,\dots,\psi_m)\in C^1_c(\Omega,\R^m)$, i.e. is 
compactly  supported  in $\Omega$. Throughout the paper we will frequently use the fact that $Q_U$ is also well defined on functions in $H^1_0(\Omega,\R^m)$ which vanish outside a compact set.\\
\begin{definition}\label{d-morse-index}
Let $U$ be a $C^2(\Omega,\R^m)$ solution of \eqref{1} and \eqref{2}. We say that
\begin{itemize}
\item $U$ is linearized stable, or it has zero Morse index, if $Q_U(\psi,\Omega)\geq 0$ for any $\psi \in C^1_c(\Omega,\R^m)$;
\item $U$ has (linearized) Morse index equal to the integer $\mu=\mu(U)\geq 1$, if $\mu$ is the maximal dimension of a subspace $X\subset C^1_c(\Omega,\R^m)$ such that
$Q_U(\psi,\Omega)<0$
for any $\psi \in X\setminus\{0\}.$
\end{itemize}
\end{definition}
\begin{remark}\rm 
This definition of linearized stability, given through the quadratic form, implies in bounded domains that the principal eigenvalue of the linearized operator is nonnegative. This is proved in \cite{DAPA} (see Proposition 2.7 (iv)) by considering a symmetrized system, as defined in Section 2 below.
\end{remark}
\vspace{10pt}
\noindent Finally we recall some coupling conditions for systems.\\
\begin{definition}\label{d-sist-cooperativo}
\begin{itemize}
\item We say that the system \eqref{1} is cooperative or weakly coupled in an open set $D \subseteq \Omega $  if
$$  \frac {\de f _i} {\de u _j} (|x|, u_1, \dots , u_m) \geq 0 \quad \text{ for any } \; (x,u_1,\dots ,u_m ) \in  D  \times  \R^m
$$
for any  $i,j=1,\dots ,m$ with  $ i \neq j$.
\item  We say that the system \eqref{1} is fully  coupled along a solution $U$ in an open set $D \subseteq \Omega $ if  it is cooperative in $D $ and,
 in addition, for any $ I,J \subset \{1, \dots ,m \}$ such that $I \neq \emptyset $, $J \neq \emptyset $, $I \cap J = \emptyset $, $I \cup J = \{1, \dots ,m \} $ there exist $i_0 \in I$, $j_0 \in J $ such that
$$\text{meas }\left(\{ x \in D :   \frac {\de f _{i_0}} {\de u _{j_0}} (|x|, U(x)) > 0 \}\right) >0.$$
 \end{itemize}
\end{definition}
\begin{remark}\rm 
Let us remark that the full coupling condition is very close to the standard notion of irreducibility for systems (see for example  \cite{Am}).
 \end{remark}
\vspace{10pt}
Let $e\in S^{N-1}$ be a direction, i.e. $e\in \R^{N}$, $|e|=1$, and let us define the set
\begin{equation}\label{1.*}
\Omega(e)=\{x\in \Omega\, :\, x\cdot e> 0\}.
\end{equation}
\noindent Our main symmetry results are contained in the following theorems
\begin{theorem}\label{tfconvessa}
Let $U$ be a solution of \eqref{1} and \eqref{2} such that $|\na U|\in L^2(\Omega)$  and Morse index $\mu(U)\leq N$.
Assume that:
\begin{itemize}
\item[i)] The system is fully coupled along $U$ in $\Omega(e)$ for any $e\in S^{N-1}$.
\item[ii)] For any $i,j=1,\dots,m$ $\frac{\de f_i}{\de u_j}(|x|,u_1,\dots,u_m)$ is nondecreasing in each variable $u_k$, $k=1,\dots,m$ for any $x\in \Omega$.
\item[iii)] If $m\geq 3$, for any $i\in \{1,\dots,m\}$, $f_i(|x|,u_1,\dots,u_m)=\sum_{\substack{k=1\\k\neq i}}^m g_{ik}(|x|, u_i,u_k)$ where $g_{ik}\in C^{1,\a}([0,+\infty)\times \R^2)$.
\end{itemize}
Then $U$ is foliated Schwarz symmetric.
\end{theorem}
\vspace{10pt}
\begin{theorem}\label{tf'convessa}
Let $U$ be a solution of \eqref{1} and \eqref{2} such that $|\na U|\in L^2(\Omega)$ and Morse index $\mu(U)\leq N-1$.
Assume that:
\begin{itemize}
\item[i)] The system is fully coupled along $U$ in $\Omega$.
\item[ii)] For any $i,j=1,\dots,m$ the function $\frac{\de f_i}{\de u_{j}}(|x|,U)$ is convex in $U$:
$$\frac {\de f_i}{\de u_j}(|x|,tU(x_1)+(1-t)U(x_2))\leq t\frac{\de f_i}{\de u_j}(|x|,U(x_1))+(1-t)\frac{\de f_i}{\de u_j}(|x|,U(x_2))$$
for any $x,x_1,x_2\in \Omega$ and for any $t\in [0,1]$.
\end{itemize}
Then $U$ is foliated Schwarz symmetric.
\end{theorem}
\vspace{10pt}
\noindent Theorem \ref{tfconvessa} extends the main result of \cite{DAPA} to unbounded radial domains, while Theorem \ref{tf'convessa} extends the symmetry result of \cite{DGP} to the same unbounded domains.\\[.2cm]
\begin{remark}\rm 
We remark that the bound $\mu(U)\leq N$ is necessary for the foliated Schwarz symmetry to hold. This can be seen, for example, in the scalar case and for bounded domains, considering the $(N+2)$th eigenfunction $w$ of the laplacian in $H^1_0(\Omega)$, if $\Omega $ is the ball in $\R^N$. Indeed 
$\mu(w) = N+1$ and $w$ is not foliated Schwarz symmetric. \\
Let us also note that our results do not require solutions to be bounded and neither to belong to $H^1_0(\Omega)$, but only that $|\na U|\in L^2(\Omega)$.
\end{remark}
\vspace{10pt}

\noindent The proofs of Theorem \ref{tfconvessa} and \ref{tf'convessa} are technically quite complicated. However we want to point out that a crucial point is to have suitable sufficient conditions for the Schwarz symmetry, namely the ones contained in Section \ref{se:3}, in particular Proposition \ref{rotating-plane}.\\
Moreover, in the case of Theorem \ref{tfconvessa} and \ref{tf'convessa}, to bypass the difficulty of dealing with a nonselfadjoint linearized operator we use, as in \cite{DAPA}, the linear operator associated with the symmetric part of the jacobian matrix $J_F$ of $F$ which is selfadjoint and to which the same quadratic form \eqref{forma-quadratica-intr} corresponds.\\[.2cm]
\noindent Note that if the system is of gradient type, i.e. $F=\mathit{grad} (g)$, for some scalar function $g$ (see \cite{deF2}) then the quadratic form 
corresponds to that generated by the second derivative of a suitable associated functional and hence the linearized operator is selfadjoint. However this is not the case for many interesting systems as the so-called ``hamiltonian-systems''(\cite{deF2},\cite{DAPA}).\\[.2cm]
\noindent The two above symmetry theorems can be applied to different kind of systems and solutions. In the first one, the hypothesis $ii)$ which implies that each $f_i$ is convex with respect to each variable $u_j$, $i,j=1,\dots,m$, seems, in some cases, suitable for positive solutions. Moreover Theorem \ref{tfconvessa} applies to solutions with Morse index up to the dimension $N$ and, for $m\geq 3$, requires an additional hypothesis.\\
\noindent Instead Theorem \ref{tf'convessa} applies more generally to sign changing solutions and does not need extra-assumptions for $m\geq 3$. On the contrary the hypothesis on the Morse index is more restrictive since it requires $\mu(U)\leq N-1$.\\[.2cm]
\noindent As a consequence of the proofs of the symmetry results we derive a necessary condition to be satisfied by a solution which could be used to establish some nonexistence results
\begin{theorem} \label{corollario1}
Assume $U$ is a nonradial solution of \eqref{1} and \eqref{2} and either
\begin{itemize}
\item[a)] $U$ has Morse index one and satisfies the assumptions of Theorem \ref{tfconvessa} or of Theorem \ref{tf'convessa};\\
or
\item[b)] the assumptions of Theorem  \ref{tf'convessa} are satisfied and there exist $i_0, j_0\in \{1,\dots,m\}$ such that $\frac{\de f_{i_0}}{\de u_{j_0}}(|x|,S)$ satisfies the following strict convexity assumption:
\begin{equation}  \label{strictconvexity}
\frac {\de f_{i_0}} {\de u_{j_0}}(|x|,t  S'+(1-t) S'') <  t \frac {\de f_{i_0}} {\de u_{j_0}}(|x|, S')+(1-t)\frac {\de f_{i_0}} {\de u_{j_0}}(|x|, S'')
  \end{equation}
for any $t \in (0,1) $, whenever    $x \in \Omega $ and  $S'=(s'_1,\dots,s'_m)$, $ S''=(s''_1,\dots,s''_m) \in \R^m$  satisfy $s'_k \neq s'' _k $ for any $k \in \{1, \dots ,m \}$.
\end{itemize}
Then necessarily
\begin{equation} \label{superfullycoupling1}
 \sum_{j=1}^m  \frac {\de f _i}{\de u _j}(r, U(r, \theta)) \frac {\de u_j } {\de \theta  }(r, \theta )=
 \sum_{j=1}^m  \frac {\de f _j}{\de u _i}(r, U(r, \theta))   \frac {\de u_j } {\de \theta  }(r, \theta )
\end{equation}
for any    $i=1,\dots ,m $, with $(r, \theta)$ as in Definition \ref{d-schwarz-symmetry}.
In particular,  if $m=2$, from \eqref{superfullycoupling1} we derive that
 \begin{equation} \label{superfullycoupling2}
 \frac {\de f _1}{\de u _2}(|x|, U(x))=  \frac {\de f _2}{\de u _1}(|x|, U(x)) \; , \quad \text{for any }x \in \Omega \;.
 \end{equation}
 \end{theorem}
\vspace{10pt}
\noindent The symmetry results of the previous theorems hold, in particular, for stable solutions of \eqref{1} and \eqref{2}. However for these solutions (as in the scalar case) we easily obtain the radial symmetry without requiring any hypothesis on the nonlinearity. Moreover, if $\Omega=\R^N$ in the autonomous case we get that stable solutions must be constant. 

\begin{theorem}\label{soluzioni-stabili}
Every linearized stable solution of \eqref{1} and \eqref{2} such that $|\na U|\in L^2(\Omega)$ is radial. If, in addition, $\Omega=\R^N$ and $F=F(U)$, i.e. $F$ does not depend on $|x|$, then $U$ is constant.
\end{theorem}
\begin{remark}\rm 
This result is analogous to the one for scalar equations obtained in \cite{GPW}. We observe that our definition of linearized stability is stronger than the one used in \cite{FG} to get a nonexistence result for H\'enon-Lane-Emden systems.
\end{remark}
\noindent Finally, as corollary of the symmetry theorems we get some nonexistence results, analogous to those obtained in the scalar case (see \cite{GPW}), but under the stronger assumptions of Theorems \ref{corollario1}.
\begin{theorem}\label{tnonesistenza1}
Assume that one among  assumptions $a)$ and $b)$   of Theorem \ref{corollario1} holds.
If $\Omega=\R^N$ and $F=F(U)$, i.e. $F$ does not depend on  $|x|$, then there are no sign changing solutions of \eqref{1} such that
$$ U(x) \to 0 \quad \text{ as }|x|\to +\infty.$$
\end{theorem}
\begin{theorem}\label{tnonesistenza2}
Assume that one between  assumptions $a)$ and $b)$   of Theorem \ref{corollario1} is satisfied.
If $\Omega=\R^N\setminus B$ and $F=F(U)$, i.e. $F$ does not depend on  $|x|$, then there are no positive solutions of \eqref{1} \eqref{2} such that
$$ U(x) \to 0 \quad \text{ as }|x|\to +\infty.$$
\end{theorem}
\noindent For systems of gradient type the previous results can be improved in the sense that they hold under the assumptions of Theorem \ref{tfconvessa} or Theorem \ref{tf'convessa}.
\begin{theorem}\label{tnonesistenza3}
Assume that either the  assumptions  of Theorem \ref{tfconvessa} or of Theorem \ref{tf'convessa} hold. If the system \eqref{1} is of gradient type and $F$ does not depend on  $|x|$,
then there are neither sign changing solutions of \eqref{1} in $\R^N$ nor positive solutions of \eqref{1} and \eqref{2} in $\R^N\setminus B$ such that
$$ U(x) \to 0 \quad \text{ as }|x|\to +\infty.$$
\end{theorem}
We observe that the case of the system of gradient type is easier. Indeed for this kind of  systems the existence of a positive solution of the linearized equation ensures, as in Lemma 2.1 of \cite{GPW}, the positivity of the quadratic form associated to it, and, in some sense, the validity of the maximum principle. This is not true anymore for systems which are not of gradient type. In the case of bounded domains (see \cite{DAPA} and \cite{DGP}), to overcome this difficulty we used the principal eigenvalue which would not help for unbounded domains, while here we use a new sufficient condition for the foliated Schwartz symmetry.\\   
\noindent We conclude with a few remarks on the range of applicability of our theorems.\\
All our results require some information on the Morse index of the solution. If the system is of gradient type, as recalled, the quadratic form corresponds to that generated by the second derivative of a suitable associated functional. So, often, variational methods, used to find solutions, also carry information on the Morse index (see \cite{deF2}),(see  \cite{MMP}, for an example). A standard case is given by solutions obtained by the Mountain Pass theorem.\\
\noindent If the system is not of this type it could happen that the second derivative of functionals associated to the variational formulation of the system are strongly indefinite.\\
As explained and showed in \cite{DAPA} this does not mean that solutions do not have finite (linearized) Morse index. We believe that, in general, more investigation about stability properties of solutions of systems should be done.\\
We also think that the result of Theorem \ref{corollario1} is interesting and could give some new understanding of systems which have or do not have solutions and are not of gradient type.\\[.5cm]
\noindent The outline of the paper is the following. 
In Section \ref{se:2} we state or prove some preliminary results while in Section \ref{se:3} we show sufficient conditions for the foliated Schwarz symmetry. In Section \ref{se:4} we give the proofs of the two symmetry results Theorem  \ref{tfconvessa} and Theorem \ref{tf'convessa}. Finally Section \ref{se:5} is devoted to the remaining theorems.

\section{\textbf{Notations and preliminary results}}\label{se:2}
\noindent We fix some general notation. Throughout the paper, $B_R$ denotes the ball in $\R^N$ with radius $R>0$ centered at the origin. 
If $D$ is a domain, we denote by $C^1_c(D,\R^m)$ the space of all $C^1$-functions from $D$ to $\R^m$ compactly supported in
$D$.\\[.20pt]
\noindent We start by recalling some statements about linear systems.\\
\noindent Let $D$ be any smooth domain of $\R^N$ and let $A(x)$ be an $m\times m$ matrix defined in $D$, i.e. $A(x)=\left( a_{ij}(x)\right)_{i,j=1}^m$ and $a_{ij}(x)\in L^{\infty}_{loc}(D)$. We consider the linear elliptic system
\begin{equation}\nonumber
\left\{\begin{array}{ll}
-\Delta \psi+A(x)\psi=0 &\text{ in }D\\
\psi=0&\text{ on }\de D
\end{array}
\right.
\end{equation}
i.e the system
\begin{equation}\label{linear-sistem}
\left\{\begin{array}{ll}
-\Delta \psi_1+\sum_{j=1}^m a_{1j}(x)\psi_j=0 &\text{ in } D\\
\dots\dots\\
-\Delta \psi_m+\sum_{j=1}^m a_{mj}(x)\psi_j=0 &\text{ in } D\\
\psi_1=\dots=\psi_m=0&\text{ on }\de D.
\end{array}
\right.
\end{equation}
We say that a function $\psi \in H^1_0(D,\R^m)$ is a weak solution of \eqref{linear-sistem} if
$$\int_D \sum_{i=1}^m \na \psi_i\cdot \na \phi_i+\sum_{i,j=1}^m a_{ij}(x)\psi_j\phi_i\, dx=0$$
for any $\phi\in C^1_c(D,\R^m)$. 
We will denote by $(-,-)_{L^2}$ the scalar product in $L^2(D,\R^m)$, i.e. $(\psi,\phi)^2_{L^2}=\int_D \sum_{i=1}^m \psi_i \phi_i\, dx$ and by $\na \psi \cdot \na \phi=\sum_{i=1}^m\na \psi_i \cdot \na \phi_i $ for any vector valued function $\psi, \phi \in C^1(D,\R^m)$.
We recall the definition of weakly and fully coupled for this kind of systems.
\begin{definition}\label{coupled}
The linear system \eqref{linear-sistem} is said to be
\begin{itemize}
\item {\rm cooperative } or {\rm weakly coupled} in $D$ if $a_{ij}(x)\leq 0$ a.e. in $D$ whenever $i\neq j$;
\item {\rm fully coupled} in $D$ if it is weakly coupled in $D$ and for any $I,J\subset \{1,\dots,m\}$ such that $I,J\neq \emptyset$, $I\cap J=\emptyset$, $I\cup J=\{1,\dots,m\}$ there exists $i_0\in I$ and $j_0\in J$ such that ${\it{meas}}\left(\{x\in D\, :\, a_{i_0j_0}(x)<0\}\right)>0.$
\end{itemize}
\end{definition}
\noindent For any scalar function $g$, we set $g^+= \max\{g,0\}$ and $g^-=
\min\{g,0\}$. Similarly, for any vector-valued function $W=(w_1,\dots,w_m)$ we set $W^+=(w_1^+,\dots,w_m^+)$ and $W^-=(w_1^-,\dots,w_m^-)$. Here and in the sequel, inequalities involving vectors should be understood to hold componentwise, for example $\psi=(\psi_1,\dots,\psi_m)\geq 0$ means $\psi_i\geq 0$ for any index $i=1,\dots,m$.\\
\noindent We recall some known facts about Maximum Principle for systems, see \cite{deF2}, \cite{deFM}, \cite{Si} and reference therein for more details.
\begin{theorem}[Strong Maximum Principle and Hopf's Lemma]\label{SMP}
 Suppose that the linear system \eqref{linear-sistem} is fully coupled in $D$  and $U=(u_1, \dots , u_m) \in C^1 (\overline {D};\R^m)$  is a weak solution of \eqref{linear-sistem}.
If $U \geq  0 $ in $D $, then either $U \equiv 0 $ in $D$ or $U>0 $ in $D $. In the latter case if $P \in \partial D $ and $U(P)=0$ then $\frac {\de U }{\de \nu }(P) < 0$, where $\nu $ is the unit exterior normal vector at $P$.
\end{theorem}

\begin{definition} We say that the maximum principle holds for the operator $- \Delta + A(x)$ in an open set  $D \subseteq \Omega $ if  any $U \in H^1 (D,\R^m)$ such that $U \leq  0 $ on $ \de D $ (i.e. $U^{+} \in H_0^1 (D,\R^m)$ )  and such that $- \Delta U + A(x)U \leq 0 $ in $D $ (i.e.
$\int_{D}\sum_{i=1}^m \nabla u_i \cdot  \nabla \psi_i + \sum_{i,j=1}^m a_{ij}(x) u_j \psi_i \, dx \leq 0 $
for any nonnnegative  $\psi \in H_0^1 (D,\R^m) $ )
 satisfies $U \leq 0 $ a.e. in $D$.
\end{definition}
\begin{theorem}  \label{p-di-max-domini-piccoli}
There exists $\delta >0 $, depending on $A(x)$, such that for any subdomain $D \subseteq \Omega $ the maximum principle holds for the operator $- \Delta + A(x)$  in $D \subseteq \Omega $ provided $\mathit{meas}(D)  \leq \delta$.
\end{theorem}
\noindent We refer to \cite{DAPA} for a general formulation and a proof of Theorem \ref{SMP} and Theorem \ref{p-di-max-domini-piccoli}.\\
\noindent Given the linear system \eqref{linear-sistem} in $D\subseteq \Omega$, we can associate with it the quadratic form
\begin{equation}\label{f-quadr}
Q_{A}(\psi,D):= \int_D\sum_{i=1}^m |\na \psi_i|^2+\sum_{i,j=1}^m a_{ij}(x) \psi_i\psi_j\, dx
\end{equation}
for any $\psi\in C^1_c(D,\R^m)$. By density, the quadratic form $Q_A$ is well defined also if $\psi $ is in $H^1_0 (D, \R^m)$ and vanishes a.e. outside a bounded set. \\
\noindent We now consider the (symmetric) linear system defined by:
\begin{equation}\label{linear-simmetric-sistem}
\left\{\begin{array}{ll}
-\Delta \psi_1+\frac 12 \sum_{j=1}^m \left(a_{1j}(x)+a_{j1}(x)\right)\psi_j=0 &\text{ in } D\\
\dots\dots\\
-\Delta \psi_m+\frac 12\sum_{j=1}^m \left(a_{mj}(x)+a_{jm}(x)\right)\psi_j=0 &\text{ in } D\\
\psi_1=\dots=\psi_m=0&\text{ on }\de D,
\end{array}
\right.
\end{equation}
and observe that to \eqref{linear-sistem} and \eqref{linear-simmetric-sistem} corresponds  the same quadratic form \eqref{f-quadr}. Obviously if the matrix $A(x)$ of \eqref{linear-sistem} is symmetric then \eqref{linear-sistem} and \eqref{linear-simmetric-sistem} coincide. Moreover we remark that if \eqref{linear-sistem} is weakly (fully) coupled in $D$ then the same holds for the symmetric system \eqref{linear-simmetric-sistem}.\\
Next we consider the  bilinear symmetric  form associated to \eqref{f-quadr}:
\begin{equation}\label{P_A}
P_A(\psi,\phi, D):=\int_D \sum_{i=1}^m\na \psi_i\cdot \na \phi_i+\frac 12\sum_{i,j=1}^m \left(a_{ij}(x)+a_{ji}(x))\right)\psi_j\phi_i\, dx
\end{equation}
for any $\psi,\phi\in C^1_c(D,\R^m)$ (but also for $\psi,\phi\in H^1_0(D,\R^m)$ which vanish outside a bounded set). 
Note that if the quadratic form \eqref{f-quadr} is positive semidefinite then the bilinear symmetric form \eqref{P_A} defines a scalar product and hence the Cauchy-Schwarz inequality holds, i.e.
$$\left(P_A(\psi,\phi, D)\right)^2\leq P_A(\psi,\psi, D)P_A(\phi,\phi, D)=Q_{A}(\psi,D)Q_{A}(\phi,D)$$
for any $\psi,\phi\in H^1_0(D,\R^m)$ vanishing a.e. outside a bounded set.\\
If $D$ is
 bounded  then the symmetric system \eqref{linear-simmetric-sistem}
has a sequence of Dirichlet eigenvalues $\l_k$ such that $\l_k\to +\infty$ as $k\to +\infty$ and a corresponding sequence of eigenfunctions $Z^k\in H^1_0(D,\R^m)$, i.e. functions that satisfy
\begin{equation}\nonumber
\left\{
\begin{array}{ll}
-\Delta Z^k_i +\sum_{j=1}^m \frac 12 \left( a_{ij}(x)+a_{ji}(x))\right)Z^k_j =\l_kZ^k_i& \hbox{ in } D\\
Z^k=0 &\hbox{ on } \de D.
\end{array}
\right.
\end{equation}
We denote by $\l_k^s(L_A,D)$ these symmetric eigenvalues.\\
The first symmetric eigenfunction, i.e. the function associated with $\l_1^s(L_A,D)$ does not change sign in $D$ and the first eigenvalue is simple, i.e. up to a scalar multiplication there is only one eigenfunction corresponding to $\l_1^s(L_A,D)$.\\
See \cite{DAPA} for a careful study of the properties of these symmetric eigenvalues.
\begin{lemma}\label{ln2}
Let $L_A$ and  $Q_A$ be defined as above. If
\begin{equation}\label{neg}
\inf_{\psi \in C^1_c(\Omega(e),\R^m)}Q_A(\psi,\Omega(e))<0 \quad \text{ for any } e\in S^{N-1}
\end{equation}
with $\Omega(e)$ as in \eqref{1.*}, 
then there exists $\tilde R > 0$ (depending on $A(x)$) such that, for any $e \in S^{N-1}$ and for any $R\geq\tilde{R}$, the first symmetric eigenvalue $\l_1^s( L_A,\Omega (e)\cap B_R )$ of the operator $L_A$ in the
domain $B_R \cap \Omega(e)$ with zero Dirichlet boundary conditions is negative.
\end{lemma}
\begin{proof}
Arguing by contradiction we assume there exists a sequence of directions $e_n\in S^{N-1}$ and a sequence of radii $R_n\to +\infty$ such that
\begin{equation}\label{negaz-Lu}
\l_1^s (L_A,\Omega(e_n)\cap B_{R_n})\geq 0.
\end{equation}
Up to a subsequence $e_n\to \tilde e\in S^{N-1}$ and it holds
\begin{equation}\label{fxxx}
\inf_{\psi \in C^1_c(\Omega(\tilde{e}),\R^m)}Q_A(\psi,\Omega(\tilde{e}))\geq 0
\end{equation}
contradicting \eqref{neg}. Indeed, if \eqref{fxxx} does not hold there should exist a function  $\psi\in C^1_c(\Omega(\tilde e ),\R^m)$ such that $Q_A(\psi,\Omega(\tilde e))<0$ and this would  imply that
$\l_1^s(L_A, \Omega(\tilde e)\cap B_{R_n})<0$ for  $n$ sufficiently large. The continuity of the first symmetric eigenvalue implies that  $\l_1^s(L_A, \Omega(e_n)\cap B_{R_n})<0$ for $n$ large enough contradicting \eqref{negaz-Lu}.
\end{proof}
\noindent We introduce some more notation. 
For a unit vector  $e\in
S^{N-1}$  we consider the hyperplane   $T(e)=\{x\in\R^N,\, : \,
x\cdot e=0\}$.
We write $\sigma_e \,:\,\Omega\mapsto \Omega$ for the reflection
with respect to $T(e)$, that is $\sigma_e(x)=x-2(x \cdot e)e$ for every
$x\in \Omega$, and denote by $U^{\sigma_e}=(u_1^{\sigma_e},\dots,u_m^{\sigma_e})$ the function $U\circ \sigma_e$. Note that $T(-e)=T(e)$ and
$\Omega(-e)=\sigma_e(\Omega(e))=-\Omega (e) $ for every $e\in
S^{N-1}$. Finally
for a given direction  $e\in S^{N-1}$ let us denote by
$W^e=(w_1,\dots,w_m)$ the difference between $U$ and its reflection with respect to the hyperplane $T(e)$,
i.e.  $W^e(x)=U(x)-U(\sigma_e(x))$.\\
Obviously the function $W^e
$ satisfies the linear system $-\Delta W^e=F(|x|,U)-F(|x|,U^{\sigma_e})$ in $\Omega$ and in $\Omega(e)$. This system can be written as a linear system in many ways.
Indeed we need at least two different formulations of it to deal with the different hypotheses of Theorem \ref{tfconvessa} and Theorem \ref{tf'convessa}. \\
First, it is standard to see that the function $W^e$ satisfies in $\Omega(e)$ and in $\Omega$ the system
\begin{equation}\label{diff-2}
\left\{\begin{array}{ll}
-\Delta W^e+ B^e(x)W^e=0 & \text{ in }\Omega(e)\\
W^e=0 & \text{ on }\de \Omega(e)
\end{array}\right.
\end{equation}
where $B^e(x)=\left( b_{ij}^e(x)\right)_{i,j=1}^m$ with
\begin{equation}\label{2.14'}
b_{ij}^e(x)=-\int_0^1 \frac{\de f_i}{\de u_j}\big( |x|, tU(x)+(1-t)U(\sigma _e(x)\big)\, dt.
\end{equation}
Now, we can write
\begin{align*}
&f_i(|x|,U(x))-f_i(|x|,U(\sigma_e(x)))\\
&=f_i(|x|,u_1(x),\dots,u_m(x))-f_i(|x|,u_1^{\sigma_e}(x),u_2(x),\dots,u_m(x))\\
&+f_i(|x|,u_1^{\sigma_e}(x),u_2(x),\dots,u_m(x))-f_i(|x|,u_1^{\sigma_e}(x),u_2^{\sigma_e}(x),\dots,u_m(x))\\
&\dots\dots\\
&+f_i(|x|,u_1^{\sigma_e}(x),\dots,u_{m-1}^{\sigma_e}(x),u_m(x))-f_i(|x|,u_1^{\sigma_e}(x),\dots,u_m^{\sigma_e}(x))\\
&=-\left(\tilde b_{i1}^e(x)w_1+\dots+\tilde b_{im}^e(x)w_m\right)
\end{align*}
where
\begin{equation}\label{2.14''}
\tilde b_{ij}^e(x)=-\int_0^1 \frac{\de f_i}{\de u_j}\big( |x|, u_1^{\sigma_e},\dots,u_{j-1}^{\sigma_e},tu_j(x)+(1-t)u_j(\sigma _e(x)),u_{j+1},\dots,u_m\big)\, dt.
\end{equation}
This implies that the function $W^e$ satisfies in $\Omega(e)$ and in $\Omega$ the linear system
\begin{equation}\label{diff-1}
\left\{\begin{array}{ll}
-\Delta W^e+\widetilde B^e(x)W^e=0 & \text{ in }\Omega(e)\\
W^e=0 & \text{ on }\de \Omega(e)
\end{array}\right.
\end{equation}
where $\widetilde B^e(x)=\left( \tilde b_{ij}^e(x)\right)_{i,j=1}^m$.\\
We collect some properties of these linear systems in the following lemma.
\begin{lemma}\label{l-differenza}
Let $U$ be a solution of \eqref{1} and \eqref{2} and $e$ any direction, $e\in S^{N-1}$.
\begin{itemize}
\item[i)] Assume that hypotheses $i),ii)$ and $iii)$ of Theorem \ref{tfconvessa} hold. Then, for any $x\in \Omega$, we have
\begin{align}
\label{b_ii} \tilde b_{ii}^e(x)\geq &-\frac {\de f_i}{\de u_i}(|x|,U(x))\quad \text{ if }u_i(x)\geq u_i^{\sigma_e}(x),\\
\label{b_ij} \tilde b_{ij}^e(x)\geq &-\frac {\de f_i}{\de u_j}(|x|,U(x))\quad \text{ if }u_i(x)\geq u_i^{\sigma_e}(x)\, ,\,u_j(x)\geq u_j^{\sigma_e}(x)
\end{align}
in particular, if $u_i(x)= u_i^{\sigma_e}(x)$, $u_j(x)= u_j^{\sigma_e}(x)$ then $\tilde b_{ij}^e(x)=-\frac {\de f_i}{\de u_j}(|x|,U(x))$. Moreover the system \eqref{diff-1} is fully coupled in  $\Omega$ and in $\Omega(e)$.

\item[ii)]If the system \eqref{1} is fully coupled along $U$ in $\Omega$ then the linear system \eqref{diff-2}  is fully coupled in $\Omega$ and in $\Omega(e)$ for any $e\in S^{N-1}$. If also hypothesis $ii)$ of Theorem \ref{tf'convessa} holds, and we let, for any direction $e\in S^{N-1}$
\begin{equation} \label{b-es}
b_{ij}^{e,s}(x)=\frac 12 \left(\frac{\de f_i}{\de u_j}( |x|, U(x))+\frac{\de f_i}{\de u_j}( |x|, U(\sigma_e(x)) )\right)
\end{equation}
then, for any $i,j=1,\dots,m$ and $x\in \Omega$, we have
\begin{equation}\label{b>b-es}
b_{ij}^e(x)\geq b_{ij}^{e,s}(x)
\end{equation}
and the linear system
\begin{equation}\label{equation-b-es}
-\Delta \psi+B^{e,s}(x)\psi=0
\end{equation}
is fully coupled  in $\Omega$ and in $\Omega(e)$ as well for any $e\in S^{N-1}$, where $B^{e,s}(x)=\left( b_{ij}^{e,s}(x)\right)_{i,j=1}^m$. Finally, if $U$ is symmetric with respect to the hyperplane $T(e)$ then $b_{ij}^e(x)= b_{ij}^{e,s}(x)=\frac{\de f_i}{\de u_j}( |x|, U(x))$ for any $i,j=1,\dots,m$.
\end{itemize}
\end{lemma}

\noindent The proof of this lemma is the same as in the case of bounded domains, see Lemma 3.1 of \cite{DAPA} for case $i)$ and Lemma 3.1 of \cite{DGP} for case $ii)$.
Throughout the paper we will denote by
\begin{equation}\label{Q-es}
Q_{es}(\psi,D):=\int_D \sum_{i=1}^m |\na \psi_i|^2 +\sum_{i,j=1}^m b_{ij}^{e,s}(x)\psi_j\psi_i\, dx
\end{equation}
the quadratic form associated with the linear system \eqref{equation-b-es}, and by
\begin{equation}\label{P-es}
P_{es}(\psi,\phi,D):=\int_D \sum_{i=1}^m \na \psi_i\cdot \na \phi_i +\frac 12\sum_{i,j=1}^m \left(b_{ij}^{e,s}(x)+b_{ji}^{e,s}(x)\right)\psi_j\phi_i\, dx
\end{equation}
the corresponding bilinear symmetric form.\\
Now we recall the following result.
\begin{lemma}\label{l2.1}
Assume that $|\na U|\in L^2 (\Omega)$. 
Then, for any $e \in S^{N-1}$ and for any $j=1,\dots,m$,  we have  
\begin{equation}
\frac 1{R^2}\int_{B_{2R}\setminus B_R} w_j ^2\, dx \ra  0 \quad \hbox{ as }R \to  \infty
\end{equation}
where $w_j$ are the components of the function $W^e$.
\end{lemma}
\noindent See Lemma 2.2 in \cite{GPW} for a detailed proof.
\begin{lemma}\label{ln1}
Let $U$ be a solution of \eqref{1} and \eqref{2}, such that the system \eqref{1} is fully coupled along $U$ in $\Omega(e)$ for any $e\in S^{N-1}$. Then there exists $\bar R>0$ such that the system \eqref{1} is fully coupled  along $U$ in $\Omega(e)\cap B_R$  for any $R>\bar R$ and for any $e\in S^{N-1}$.
\end{lemma}

\begin{proof}
Assume, by contradiction, that there exists a sequence of radii $R_n\to +\infty$, a sequence of directions $e_n\in S^{N-1}$ and a sequence of subsets $I_n\subset \{1,\dots,m\}$ such that
\begin{equation}\label{f1}
\mathit{meas}\left(\left\{ x\in \Omega(e_n)\cap B_{R_n}, \frac{\de f_{i_n}}{\de u_{j_n}}(|x|,U)>0\right\}\right)=0
\end{equation}
for any  $i_n\in I_n$ and for any $j_n\in \{1, \dots, m\}\setminus I_n$.\\
Since $I_n\subset \{1, \dots, m\}$ there exists $I\subset \{1, \dots, m\}$, $I\neq \emptyset$ and a subsequence $s(n)$ such that $I_{s(n)}=I$ for any $n\in \N$. Up to a subsequence $e_{s(n)}\to e\in S^{N-1}$ and $\Omega(e_{s(n)})\cap  B_{R_{s(n)}}\to \Omega (e)$. \\
The hypothesis of the fully coupling in $\Omega(e)$ implies that there exist $i_0\in I$ and $j_0\in \{1, \dots, m\}\setminus I$ such that
$$\mathit{meas}\left(\left\{ x\in \Omega(e), \frac{\de f_{i_0}}{\de u_{j_0}}(|x|,U)>0\right\}\right)>0.$$
Then there exist $x\in \Omega(e)$ and $B_{\rho}(x)\subset \Omega(e)$ such that $B_{\rho}(x)\subset \left\{ x\in \Omega(e), \frac{\de f_{i_0}}{\de u_{j_0}}(|x|,U)>0\right\}$.
By continuity $B_{\rho}(x)\subset\Omega(e_{s(n)})\cap B_{R_{s(n)}}$ contradicting \eqref{f1}.
\end{proof}
\begin{remark}\label{lfa}\rm
The same proof of the previous Lemma shows also that, if  $U$ is a solution of \eqref{1} and \eqref{2}, such that the system \eqref{1} is fully coupled along $U$ in $\Omega$ then there exists $\bar R>0$ such that the system \eqref{1} is fully coupled along $U$ in $\Omega\cap B_R$ for any $R>\bar R$. Moreover by $ii)$ of Lemma \ref{l-differenza} we also have that the linear system  \eqref{equation-b-es} is fully coupled in $\Omega(e)\cap B_R$  for any $R>\bar R$ and for any $e\in S^{N-1}$.
\end{remark}
\noindent Let $U$ be a solution of \eqref{1} and \eqref{2} and let  $D\subseteq \Omega$. We we will denote by $L_U$ the linearized operator at $U$, i.e. $L_U=-\Delta -J_F(|x|,U)$ where $ J_F(|x|,U)$ is the jacobian matrix of $F$ computed at $U$, and we define the linearized system at $U$, i.e.
\begin{equation}\label{linearized}
\left\{
\begin{array}{ll}
-\Delta \psi -J_F(|x|,U)\psi & \hbox{ in } D\\
\psi=0 &\hbox{ on } \de D.
\end{array}
\right.
\end{equation}
Associated with system \eqref{1}, or with system \eqref{linearized}, we can consider the quadratic form
\begin{equation}\label{f-quadratica}
Q_U(\psi, D):=\int_D \sum_{i=1}^m|\na \psi_i|^2-\sum_{i,j=1}^m \frac{\de f_i}{\de u_j}(|x|,U)\psi_j\psi_i\, dx
\end{equation}
and the bilinear symmetric form
\begin{equation}\label{prod-scalare}
P_U(\psi,\phi, D):=\int_D \sum_{i=1}^m\na \psi_i\cdot \na \phi_i-\frac 12\sum_{i,j=1}^m \Big(\frac{\de f_i}{\de u_j}(|x|,U)+\frac{\de f_j}{\de u_i}(|x|,U)\Big)\psi_j\phi_i\, dx
\end{equation}
for any $\psi,\phi\in C^1_c(D,\R^m)$ (but also for $\psi,\phi\in H^1_0(D,\R^m)$ which vanish outside a bounded set).  In the sequel we will denote by $J_F^t(|x|,U)$ the transpose matrix of $ J_F(|x|,U)$.\\

For an arbitrary subset $D$ of $\R^N$ we let
$\chi_{\stackrel{}{D}}: \R^N \to \R$
denote the characteristic function of $D$.
To avoid complicated notation, we denote the restriction of $\chi_{\stackrel{}{D}}$ to
arbitrary subsets of $\R^N$ again by $\chi_{\stackrel{}{D}}$.\\

We introduce a family of cutoff functions which we will be frequently used
throughout the paper. To this aim we fix a $C^{\infty}$-function $\xi$
defined on $[0,\infty)$ such that  $0\leq \xi\leq
1$ and
\begin{equation}\nonumber
\xi(t)=\left\{
\begin{array}{ll}
1 & \hbox{ if } 0\leq t\leq 1\\
0& \hbox{ if }  t\geq 2.
\end{array}\right.
\end{equation}
For $R>0$, we then define
\begin{equation}
  \label{eq:def-xi_R}
\xi_R \in C^\infty_c(\R^N,\R),\qquad \xi_{R}(x)=\xi\left(\frac
  {|x|}R\right).
\end{equation}
We will denote the restriction of $\xi_R$ to
arbitrary subsets of $\R^N$ again by $\xi_R$.\\

\begin{lemma}\label{l2.2}
Let $U$ be a solution of \eqref{1} and \eqref{2} such that $|\na U|\in L^2(\Omega)$. Let $e\in S^{N-1}$ and 
$D \subset \Omega(e)$ be a (possibly unbounded) open set such that
$W^e_{|\de D}\equiv 0$. Then
\begin{itemize}
\item[i)] if  assumptions $i)$, $ii)$ and $iii)$ of Theorem \ref{tfconvessa} are satisfied and if $W^e\geq 0$  in $D$, we have
\begin{equation}\label{Q_u<0}
\limsup_{R\to +\infty}Q_U(W^e\chi_{D}\xi_R,\Omega(e))\leq 0
\end{equation}
where $Q_U$ is as in \eqref{f-quadratica}.
\item[ii)] if  assumptions $i)$, and $ii)$ of Theorem \ref{tf'convessa} are satisfied, we have
\begin{equation}\label{Q_es<0}
\limsup_{R\to +\infty}Q_{es}((W^e)^{\pm}\chi_{D}\xi_R,\Omega(e))\leq 0
\end{equation}
where $Q_{es}(-,D)$ is as defined in \eqref{Q-es}.
\end{itemize}
\end{lemma}
\noindent Note that since $W^e_{|\de D}\equiv 0$, the function
$W^e\chi_{D}\xi_R$  belongs to $H^1_0(\Omega(e),\R^m)$ for $R > 0$ and vanishes a.e. outside a bounded subset of $\Omega(e)$.
Hence the quadratic forms  $Q_U$
 and $Q_{es}$ are well defined on these functions.
\begin{proof}
{\rm i)} The function $W^e=(w_1,\dots,w_m)$ satisfies the linear system \eqref{diff-1} in $\Omega(e)$.
Multiplying the $i$-th equation of the linear system  \eqref{diff-1}  by $w_i \chi_{D}\xi_R^2$ and integrating on $\Omega(e)$ we get
\begin{align*}
0&=\int_{\Omega(e)} (-\Delta w_i)(w_i \chi_{D}\xi_R^2)\, dx +\sum_{j=1}^m \int_{\Omega(e)} \tilde b_{ij}^e(x) w_j (w_i \chi_{D}\xi_R^2)\, dx\\
&=\int_{D}\na w_i\cdot \na (w_i\xi_R^2)\, dx+\sum_{j=1}^m \int_{D} \tilde b_{ij}^e(x) w_j (w_i \xi_R^2)\, dx.
\end{align*}
Using \eqref{b_ii} and \eqref{b_ij}, we deduce
$$0\geq \int_{D}\na w_i\cdot \na (w_i\xi_R^2)\, dx- \sum_{j=1}^m \int_{D}\frac{\de f_i}{\de u_j}(|x|,U)(w_j \xi_R)(w_i\xi_R)\, dx$$
for any $i=1,\dots,m$.
Letting $v_R:= W^e\chi_{D}\xi_R$, we have, from the definition of $Q_U$, that
\begin{align*}
Q_U(v_R,\Omega(e))&=\sum_{i=1}^m \int_{\Omega(e)}|\na w_i\chi_{D} \xi_R|^2\, dx-\sum_{i,j=1}^m \int_{\Omega(e)}\frac{\de f_i}{\de u_j}(|x|,U)(w_j \chi_{D}\xi_R)(w_i\chi_{D}\xi_R)\, dx\\
&= \sum_{i=1}^m \int_{D}\left[w_i^2 |\na \xi_R|^2+\na w_i\cdot \na (w_i\xi_R^2)\right]\, dx-\sum_{i,j=1}^m \int_{D}\frac{\de f_i}{\de u_j}(|x|,U)(w_j \xi_R)(w_i\xi_R)\, dx\\
&=\sum_{i=1}^m\int_{D} \na w_i\cdot \na (w_i\xi_R^2)-\sum_{j=1}^m  \frac{\de f_i}{\de u_j}(|x|,U)(w_j \xi_R)(w_i\xi_R)\, dx+\sum_{i=1}^m \int_{D}w_i^2 |\na \xi_R|^2\, dx\\
&\leq \sum_{i=1}^m \int_{D}w_i^2 |\na \xi_R|^2\, dx\leq \sum_{i=1}^m \frac C{R^2}\int_{B_{2R}\setminus B_R}w_i^2\, dx\to 0,
\end{align*}
having applied Lemma \ref{l2.1}. Thus {\rm i)} is proved.\\[.3cm]
{\rm ii)} We give the proof of \eqref{Q_es<0} for $(W^e)^+$. The case of $(W^e)^-$ follows in the same way.   The function $W^e$ satisfies the linear system \eqref{diff-2} in $\Omega(e)$. Multiplying the $i$-th equation of this system by $w_i^+\chi_{D}  \xi_R^2 $ and integrating on $\Omega(e)$ we get
$$\int_{\Omega(e)} \left(\Delta w_i\right)\left( w_i^+ \chi_{D}\xi_R^2\right)\, dx=-\int_D \na w_i\cdot \na \left(w_i^+\xi_R^2\right)\, dx=\sum_{j=1}^m \int_D b_{ij}^e(x) w_j w_i^+\xi_R^2\, dx$$
for $i=1,\dots,m$, so that, letting $v_R:=(W^e)^+\chi_{D} \xi_R$, we have
$$\int_{\Omega(e)}|\na v_R|^2 \, dx=\sum_{i=1}^m \int_{D} (w_i^+)^2|\na \xi_R|^2\, dx-\sum_{i,j=1}^m \int_{D}b_{ij}^e(x) w_jw_i^+ \xi_R^2\, dx.$$
Then, the definition of $Q_{es}$ implies that
\begin{align*}
Q_{es}(v_R,&\,\Omega(e))=\int_{\Omega(e)}|\na v_R|^2 \, dx
+\sum_{i,j=1}^m \int_{\Omega(e)}b_{ij}^{es}(x) w_j^+w_i^+\chi_{D} \xi_R^2\, dx\\
&=\sum_{i=1}^m \int_{D} (w_i^+)^2|\na \xi_R|^2\, dx-\sum_{i,j=1}^m \int_{D}b_{ij}^e(x) (w_j^+-w_j^-)w_i^+ \xi_R^2\, dx\\
&+\sum_{i,j=1}^m \int_{D}b_{ij}^{es}(x) w_j^+w_i^+ \xi_R^2\, dx\\
&=\sum_{i,j=1}^m \int_{D}\left(b_{ij}^{es}(x)-b_{ij}^e(x)\right) w_j^+w_i^+ \xi_R^2\, dx\\
&+\sum_{\substack{i,j=1\\i\neq j}}^m \int_{D}b_{ij}^e(x) w_j^-w_i^+ \xi_R^2\, dx+\sum_{i=1}^m \int_{D} (w_i^+)^2|\na \xi_R|^2\, dx.
\end{align*}
Thus from Lemma \ref{l-differenza} we get that $b_{ij}^e(x)\leq 0$ for  $i\neq j$ and that  $b_{ij}^{es}(x)\leq b_{ij}^e(x)$. Then
$$Q_{es}(v_R,\Omega(e))\leq \limsup_{R\to +\infty}\frac C{R^2}\sum_{i=1}^m \int_{B_{2R}\setminus B_R}(w_i^+)^2\, dx$$
and \eqref{Q_es<0} follows again by Lemma \ref{l2.1}.
\end{proof}
\section{\textbf{Sufficient conditions for foliated Schwarz symmetry}}\label{se:3}
\begin{lemma}\label{lfss}
Let $U=(u_1,\dots,u_m)$ be a solution of \eqref{1} and \eqref{2} and assume that the hypothesis $i)$ either of Theorem \ref{tfconvessa} or of Theorem \ref{tf'convessa} holds.
If for every  $e\in S^{N-1}$ we have either $U\geq
U^{\sigma_e}$ in $ \Omega(e)$ or $U\leq U^{\sigma_e}$
in $ \Omega(e)$, then $U$ is foliated Schwarz symmetric.
\end{lemma}
\noindent The proof is exactly the same as in the case of bounded domains, see Lemma  3.2 in \cite{DAPA} and  Lemma 3.2 in \cite{DGP} as well as \cite{W} for the scalar case.\\[.5cm]
We will now describe other sufficient conditions for the foliated Schwarz symmetry of a solution $U$ of \eqref{1} and \eqref{2}.\\
To this end we begin with some  geometric considerations about cylindrical coordinates with respect to the plane $x_1x_2$.\\
Suppose $\beta\in \R$ and let $e_{\beta}=(\cos \beta,\sin \beta,0,\dots,0)$ be a unit vector in the $x_1x_2$ plane.
Then we consider as before the hyperplane $T(e_{\beta})$ and, for simplicity, we will use the notations
$\Omega_{\beta}=\Omega( e_{\beta})$, $T_{\beta}=T(e_{\beta})$, $\sigma_{\beta}=\sigma_{e_{\beta}}$.\\
Using cylindrical coordinates we will write $x=(x_1,\dots,x_N)$ as $x=(r,\theta,\tilde x)=(r\cos\theta,r\sin\theta,\tilde x)$ where $r=\sqrt{x_1^2+x_2^2}$, $\tilde x=(x_3,\dots,x_N)$ and $\theta\in [0,2\pi]$.\\
It is easy to see that  the reflection $\sigma_{\beta}$ through $T_{\beta}$ can be written as
\begin{equation}\label{lucio-1}
\sigma_{\beta}(r\cos\theta,r\sin\theta,\tilde x)=(r\cos(2\beta-\theta+\pi),r\sin(2\beta-\theta+\pi), \tilde x)
\end{equation}
(in fact $2\beta-\theta+\pi=\theta+2(\beta+\frac{\pi}2-\theta)$).
It can also be proved analytically writing the usual reflection in cartesian coordinates and using simple trigonometry formulas.\\
Of course, since the angular variable is defined up to a multiple of $2\pi$, we could also write the angular variable of the of the point $\sigma_{\beta}(r\cos\theta,r\sin\theta,\tilde x)$ as $ 2\beta-\theta-\pi$. \\
 Let us put $h_{\beta}^{\pm}(\theta)=2\beta-\theta\pm\pi$. \\
Note that if we choose an interval $[\theta_0,\theta_0+2\pi)$ to which the angular coordinate $\theta$ belongs, the images $2\beta-\theta \pm \pi$ could not belong to the same interval.
Nevertheless we observe that for a fixed $\tilde\beta\in \R$, if we take $\theta\in [\tilde \beta-\frac{\pi}2,\tilde \beta+\frac 32{\pi}]$ then $h_{\tilde\beta}^+(\theta)$ belongs to the same interval, whereas if we take $\theta\in [\tilde \beta-\frac 32{\pi},\tilde \beta+\frac{\pi}2]$ then $h_{\tilde\beta}^-(\theta)$ belongs to the same interval. More precisely we have
\begin{align}
&\label{lucio-2}\theta \in [\tilde \beta-\frac{\pi}2,\tilde \beta+\frac {\pi}2]\Rightarrow 2\tilde\beta-\theta+\pi=h_{\tilde\beta}^+(\theta)\in [\tilde \beta+\frac{\pi}2,\tilde \beta+\frac 32{\pi}],\\
&\label{lucio-3}\theta \in [\tilde \beta-\frac 32{\pi},\tilde \beta-\frac {\pi}2]\Rightarrow 2\tilde\beta-\theta-\pi=h_{\tilde\beta}^-(\theta)\in [\tilde \beta-\frac{\pi}2,\tilde \beta+\frac {\pi}2].
\end{align}
This can be easily verified evaluating $h_{\tilde\beta}^{\pm}$ in the boundary of the intervals of definition, since the mappings $h_{\tilde\beta}^{\pm}$ are decreasing.\\
Let us denote by $U_{\theta}(r,\theta,\tilde x)$ the derivative of the function $ U$ whit respect to the angular coordinate $\theta$.
\begin{proposition}\label{proposition-A}
Let $\tilde \beta\in \R$ and assume that $U$ is symmetric with respect to the hyperplane $T_{\tilde \beta}$. If $U_{\theta}\geq 0$ in $\Omega_{\tilde \beta}=[x\cdot e_{\tilde\beta}>0]$, then for any $\beta\in [\tilde \beta-\pi,\tilde \beta]$ and for any $x\in \Omega_{\beta}=[x\cdot e_{\beta}>0]$ we have $U(x)\leq U(\sigma_{\beta}(x))$, while for every $\beta\in[\tilde \beta,\tilde \beta+\pi]$ we have $U\geq U^{\sigma_{\beta}}$ in $\Omega_{\beta}$.
\end{proposition}
\noindent To prove this proposition let us first state and prove a simple lemma.
\begin{lemma}\label{lemma-lucio}
Suppose that $t_0\in ( \tilde \beta-\frac {\pi}2,\tilde \beta+\frac{\pi}2]$ and that the assumptions of Proposition \ref{proposition-A} hold. Then 
\begin{equation}\label{lucio-4}
U(r,t,\tilde x)\geq U(r,t_0,\tilde x)\quad \forall t\in [t_0,2\tilde \beta-t_0+\pi].
\end{equation}
\end{lemma}
\begin{proof}
Since $U$ is symmetric, $U_{\theta}$ is antisymmetric with respect to $T_{\tilde \beta}$. By hypothesis  $U_{\theta}\geq 0$ in $\Omega_{\tilde \beta}=[x\cdot e_{\tilde \beta}>0]=\{(r,\theta,\tilde x)\,:\, \tilde \beta-\frac{\pi}2\leq \theta\leq \tilde \beta +\frac{\pi}2\}$ so that $U_{\theta}(r,\theta,\tilde x)\leq 0$ if $\tilde \beta+\frac{\pi}2\leq \theta\leq \tilde \beta +\frac 32{\pi}$. Moreover by \eqref{lucio-2}, if $\tilde \beta-\frac{\pi}2< t_0< \tilde \beta +\frac{\pi}2$
 then $h_{\tilde \beta}^+(t_0)=2\tilde \beta-t_0+\pi\in [\tilde \beta+\frac{\pi}2,\tilde \beta +\frac 32{\pi}]$. This means that $U(r,\cdot,\tilde x)$ increases in $[t_0,\tilde \beta+\frac{\pi}2]$, then decreases in $[\tilde \beta+\frac{\pi}2, \tilde \beta +\frac 32{\pi}]$, in particular in $[\tilde \beta+\frac{\pi}2,  2\tilde \beta-t_0+\pi]$. Since $U(r,t_0,\tilde x)=U(r, 2\tilde \beta-t_0+\pi,\tilde x)$,  \eqref{lucio-4} follows.
\end{proof}
\begin{proof}[Proof of Proposition \ref{proposition-A}]
Let $\beta \in [\tilde \beta-\pi,\tilde \beta]$ and let $x\in \Omega_{\beta}=[x\cdot e_{\beta}>0]$, equivalently $x=(r\cos\theta,r\sin\theta,\tilde x)$
with $\beta-\frac{\pi}2<\theta<\beta+\frac{\pi}2$. 
We have to show that $U(x)\leq U^{\sigma_{{\beta}}}(x)$ if  \ $\beta-\frac{\pi}2<\theta<\beta+\frac{\pi}2$, \ $  \tilde \beta-\pi \leq \beta \leq \tilde \beta$. \\
Let us observe that, since $U$ is symmetric with respect to $T_{\tilde \beta}$,  
$$U^{\sigma_{{\beta}}}(x)=U(\sigma_{{\tilde \beta}}(\sigma_{{\beta}}(x)))=U(r\cos(\theta+2(\tilde \beta-\beta)),r\sin(\theta+2(\tilde \beta-\beta)),\tilde x)$$ because $2\tilde \beta-(2\beta-\theta+\pi)+\pi=\theta+2(\tilde \beta-\beta)$.\\
Let us first assume that $x\in \Omega_{ \beta}\cap \Omega_{\tilde \beta}$, i.e. $x=(r\cos \theta,r\sin\theta,\tilde x)$ with \\
$( \; \beta-\frac {\pi}2\leq  \; ) \  \tilde \beta-\frac {\pi}2<\theta< \beta+\frac {\pi}2 \; ( \; \leq  \tilde \beta+\frac {\pi}2 \; )$. Then we can apply Lemma \ref{lemma-lucio} taking $t_0=\theta$, $t= \theta+2(\tilde \beta-\beta)$: we have that $\tilde \beta-\frac {\pi}2<t_0=\theta \leq \theta+2(\tilde \beta-\beta)\leq 2\tilde \beta-\theta+\pi$.
In fact $\tilde \beta-\beta\geq 0$, and the last equality is equivalent to $\theta<\beta+\frac{\pi}2 \; )$, which is true, since $x\in \Omega_{\beta}$. So from Lemma \ref{lemma-lucio} it follows that
$$U^{\sigma_{{\beta}}}(x)=U(r,\theta+2(\tilde \beta-\beta),\tilde x)\geq U(r,\theta,\tilde x)=U(x).$$
If instead $x=(r,\theta,\tilde x)\in \Omega_{\beta}\setminus\Omega_{\tilde \beta}$, i.e. 
$ ( \; \tilde \beta -\frac 32  {\pi} \leq \; ) \beta-\frac {\pi}2<\theta\leq \tilde \beta -\frac {\pi}2 \; (\leq  \beta+\frac {\pi}2   \; )$, then by \eqref{lucio-3}
$ t_0 : = 2\tilde \beta-\theta-\pi=\in [\tilde \beta -\frac {\pi}2,\tilde \beta +\frac {\pi}2]$ and $U(x)=U(r,\theta,\tilde x)=U(r,t_0,\tilde x) $, while  $U^{\sigma_{{\beta}}}(x)=U(r,\theta+2(\tilde \beta-\beta),\tilde x)$ so that, as before the inequality follows if we show that $t_0=2\tilde \beta -\theta-\pi\leq \theta+2(\tilde \beta-\beta)\leq 2\tilde \beta-t_0+\pi=\theta+2\pi$. The inequality $\theta+2(\tilde \beta-\beta)\leq \theta+2\pi$ follows from the relation $\tilde \beta-\pi\leq \beta\leq \tilde \beta$, while $2\tilde \beta -\theta-\pi\leq \theta+2(\tilde \beta-\beta)$ is equivalent to $\theta>\beta-\frac{\pi}2$, which is true since $x\in\Omega_{\beta}$. Then by Lemma \ref{lfss} also the foliated Schwarz symmetry follows.
\end{proof}
\begin{remark}\rm
From Proposition \ref{proposition-A} it follows that if $\tilde{e}$ is a symmetry direction for $U$ and for any other direction $e'\neq \tilde{e}$ in the plane $\pi(\tilde{e},e')$ generated by $\tilde{e}$ and $e'$, using the cilyndrical coordinates with respect to $\pi$, one has $U_{\theta}\geq 0$ in $\Omega(\tilde{e})$, then $U$ is foliated Schwarz symmetric, as a consequence of Lemma \ref{lfss}. This is the strategy of Proposition \ref{pf} exploiting condition \eqref{1.10}.
\end{remark}
\begin{proposition}\label{proposition-B}
Let $\tilde \beta\in \R$ and assume that $U$ is symmetric whit respect to $T_{\tilde \beta}$. Assume further that there exists $\beta_1<\tilde \beta$ such that for any $\beta\in (\beta_1,\tilde \beta)$ we have $U\leq U^{\sigma_{\beta}}$ in $\Omega_{\beta}$. Then $U_{\theta}=\frac{\de U}{\de \theta}\geq 0$ in $\Omega_{\tilde \beta}$.
\end{proposition}
\begin{proof}
We can write the angular derivative as
$$U_{\theta}(r,\theta,\tilde x)=\lim_{\alpha\to 0^+}\frac{U(r,\theta+\alpha,\tilde x)-U(r,\theta,\tilde x)}{\alpha}.$$
With the change of variable $\alpha=2(\tilde \beta-\beta)$, $\beta=\tilde \beta-\frac{\alpha}2$, we have that $\beta\to \tilde \beta^-$. If $\alpha$ is small then $\beta\in (\beta_1, \tilde\beta)$, and, if $x\in \Omega_{\tilde \beta}$, then $x\in \Omega_{\beta}$ definitively for $\beta\to \tilde \beta^-$. Since $U^{\sigma_{\beta}}(r,\theta,\tilde x)=U(r,\theta+2(\tilde \beta-\beta),\tilde x)$ as observed, we obtain that
\begin{align*}
&U_{\theta}(r,\theta,\tilde x)=\lim_{\beta\to \tilde \beta^-}\frac{U(r,\theta+2(\tilde \beta-\beta),\tilde x)-U(r,\theta,\tilde x)}{2(\tilde \beta-\beta)}\\
&=\lim_{\beta\to \tilde \beta^-}\frac{U^{\sigma_{\beta}}(r,\theta,\tilde x)-U(r,\theta,\tilde x)}{2(\tilde \beta-\beta)}\geq 0.
\end{align*}
\end{proof}
\begin{remark}\rm
Let us remark that if if there exists a direction $\bar e\in \S^{N-1}$ such that $W^{\bar e}(x)>0$, for any $x\in \Omega(\bar e)$ and if a rotating plane argument can be applied
to the solution $U$, then the foliated Schwarz symmetry of $U$ follows from the previous propositions. Indeed, if $e'$ is any other direction, $e'\neq \bar e$, rotating the hyperplane $T({\bar e})$ we get a symmetry hyperplane 
$ T(\tilde{e})$, with $\tilde{e}$ belonging to the $2$-dimensional plane generated by $e'$ and $\bar e$.(see the proof of Proposition \ref{rotating-plane} below for details).\\
Then Proposition \ref{proposition-B} applies and by Lemma \ref{lfss} we obtain the foliated Schwarz symmetry of $U$.\\
Let us further notice that in the case of a single equation $(m=1)$ or of sysyem but in  bounded  domains the rotating plane method can be applied
without further assumptions, so the previous remark yelds the foliated Schwarz symmetry of $U$. In the scalar case, for bounded domains this was also observed in \cite{SW}, (Corollary 1.2) by a different kind of argument.\\
In the case of systems in unbounded domains we need to work under the hypothesis of Theorem \ref{tfconvessa} or Theorem \ref{tf'convessa} to perform the rotating plane method.
We will use this procedure to prove the foliated Schwarz symmetry of the solution of \eqref{1} and \eqref{2} in Proposition \ref{rotating-plane}.\\
\end{remark}
\noindent We now give another sufficient condition for foliated Schwarz symmetry.
\begin{proposition}\label{pf}
Let $U$ be a solution of \eqref{1} and \eqref{2} such that $|\na U|\in L^2(\Omega)$, and assume that the system \eqref{1} is fully coupled along $U$ in $\Omega$. Suppose further that there exists $e\in S^{N-1}$ such that  $U$ is symmetric with respect to the hyperplane $T(e)$ and
\begin{equation}\label{1.10}
\inf_{\psi\in C^1_c(\Omega(e),\R^m)} Q_U(\psi,\Omega(e))\geq 0.
\end{equation}
Then $U$ is   foliated Schwarz symmetric.
\end{proposition}
\begin{proof}
First observe that the symmetry of $U$ with respect to $T(e)$ and the coupling conditions imply that the system \eqref{1} is fully coupled along $U$ also in $\Omega(e)$.\\
We want to prove the foliated Schwarz symmetry of $U$ using Proposition \ref{proposition-A} and Lemma \ref{lfss}.\\
We follow the proof of Proposition 2.5 in \cite{GPW}.
We may assume that $e = e_2 = (0, 1,\dots , 0)$, so that $T (e) = \{x \in \R^N : x_2 = 0\}$. 
We consider an arbitrary unit vector $e' \in S^{N-1}$ different from $\pm e$. After another orthogonal
transformation which leaves $e_2$ and $T (e_2 )$ invariant, we may assume that $e' = e_{\beta }=(\cos \beta, \sin \beta, 0, . . . , 0)$ for some $\beta \in
(-\frac{\pi}2 , \frac{\pi}2 )$. 
Now we consider the cylindrical coordinates $(r,\theta,\tilde x)$ introduced before. The derivative of $U$ with respect to $\theta$ - denoted by $U_{\theta}$ - extended in the origin with $U_{\theta}(0)=0$ if $\Omega=\R^N$, satisfies the linearized system
\begin{equation}\label{sist-lineariz}
\left\{
\begin{array}{ll}
-\Delta U_{\theta}^1 =\sum_{j=1}^m \frac{\de f_1}{\de u_j}(|x|,U)U_{\theta}^j & \hbox{in } \Omega(e)\\
\dots\dots\\
-\Delta U_{\theta}^m =\sum_{j=1}^m \frac{\de f_m}{\de u_j}(|x|,U)U_{\theta}^j & \hbox{in } \Omega(e)\\
U_{\theta}^1=\dots=U_{\theta}^m=0 &\hbox{ on } \de \Omega(e).
\end{array}
\right.
\end{equation}
All we need is to show  that $U_{\theta}$ does not change sign in $\Omega (e )$.
Indeed, if this is the case, Proposition \ref{proposition-A} implies that either $U \leq U_{\sigma _{e_{\beta}}}$   or  $U \geq U_{\sigma _{e_{\beta}}}$  in $\Omega (e_{\beta})$   and by the arbitrariness of $e_{\beta} $  Lemma \ref{lfss} implies the foliated Schwarz symmetry of $U$.\\
 We first note that, by \eqref{1.10}, the bilinear form
$P_U(\psi,\phi,\Omega(e))$, defined in \eqref{prod-scalare}, defines a (semidefinite)
scalar product on $C_c^1 (\Omega(e ), \R^m)\times C_c^1 (\Omega(e ), \R^m) $, and the corresponding Cauchy-Schwarz-inequality yields:
\begin{equation}\label{1.11}
\left(
P_U (\psi, \phi,\Omega(e))\right)^2\leq  P_U (\psi, \psi,\Omega(e))P_U (\phi, \phi,\Omega(e))=Q_U (\psi,\Omega(e))Q_U (\phi,\Omega(e))
\end{equation}
for any 
$\psi , \phi$ are $H^1_0 (\Omega(e ), \R^m)$-functions vanishing a.e. outside a bounded set. We consider $\xi_R$ as
defined in \eqref{eq:def-xi_R} and, for $R > 0$, we let
$$v_R \in  H^1_0 (\Omega(e ), \R^m) ,\quad
v_R (x) = \xi_R (x)U^+_{\theta} (x)\quad \hbox{ for any } x \in \Omega(e ).$$
We also fix $\phi \in C_c^1 (\Omega(e),\R^m)$. Then \eqref{1.11} yields
\begin{align}\label{nuova}
\Big( \int_{\Omega(e)}\na U_{\theta}^+ &\cdot \na \phi -\frac 12\sum_{i,j=1}^m  \big(\frac{\de f_i}{\de u_j} (|x|,U)+\frac{\de f_j}{\de u_i} (|x|,U)\big)\left( U_{\theta}^j\right)^+ \phi^i\, dx\Big)^2\nonumber\\
&= \lim_{R\to +\infty} \left(P_U (v_R , \phi,\Omega(e))\right)^2\leq Q_U (\phi, \Omega(e)) \limsup_{R\to +\infty} Q_U (v_R , \Omega(e) ).
\end{align}
Moreover, since
\begin{align*}
\int_{\Omega(e)}&|\na v_R|^2\, dx = \sum_{i=1}^m \int_{\Omega(e)}| \na \big( \xi_R \left(U_{\theta}^i\right)^+\big)|^2 \, dx\\
& =\sum_{i=1}^m  \int_{\Omega(e)}\big(\left(U_{\theta}^i\right)^+\big)^2|\na \xi_R|^2\, dx +\sum_{i=1}^m\int_{\Omega(e)}\na \left(U_{\theta}^i\right)^+\cdot \na \big( \xi_R^2 \left(U_{\theta}^i\right)^+\big)\, dx
\end{align*}
from \eqref{sist-lineariz} we have
\begin{align*}
\int_{\Omega(e)}|\na v_R|^2\, dx &-\int_{\Omega(e)}|U_{\theta}^+|^2 |\na \xi_R|^2 \, dx=\int_{\Omega(e)}\na U_{\theta}^+\cdot \na \big( \xi_R^2 U_{\theta}^+\big)\, dx\\
&
=\sum_{i=1}^m \int_{\Omega(e) \cap \{U_{\theta}^i>0\} } \na U_{\theta}^i\cdot \na \big( \xi_R^2 \left(U_{\theta}^i\right)^+\big)\, dx\\
&= \sum_{i=1}^m \int_{\Omega(e) \cap \{ U_{\theta} ^i > 0 \} } \left[-\Delta U_{\theta}^i\right] \big( \xi_R^2 \left(U_{\theta}^i\right)^+\big)\, dx\\
&=\sum_{i=1}^m \int_{\Omega(e) \cap \{ U_{\theta}^i>0\}} \sum_{j=1}^m \frac{\de f_i}{\de u_j}(|x|,U) U_{\theta}^j \left(U_{\theta}^i\right)^+\xi_R^2 \, dx\\
&=\sum_{i,j=1}^m \int_{\Omega(e) \cap \{U_{\theta}^i>0\}} \frac{\de f_i}{\de u_j}(|x|,U)\left( \xi_R U_{\theta}^j\right) ( \xi_R \left( U_{\theta}^i)^+\right) \, dx\\
&=\sum_{i,j=1}^m \int_{\Omega(e) \cap \{ U_{\theta}^i>0 \} } \frac{\de f_i} {\de u_j} (|x|,U)\big( \xi_R \big( \big( U_{\theta}^j\big)^+-\big( U_{\theta}^j\big)^-\big)\big)v_R ^i \, dx\\
&=\sum_{i,j=1}^m \int_{\Omega(e)} \frac{\de f_i} {\de u_j}(|x|,U)v_R^j v_R^i \, dx -\sum_{\substack{i,j=1\\i\neq j}}^m \int_{\Omega(e)}\frac{\de f_i} {\de u_j} (|x|,U)\xi_R \left( U_{\theta}^j\right)^-v_R^i \, dx
\end{align*}
and since $\frac{\de f_i}{\de u_j}(|x|,U)\geq 0$ for $i\neq j$, and $\xi_R \left( U_{\theta}^j\right)^-v_R^i\geq0$ then
\begin{align*}
\int_{\Omega(e)}|\na v_R|^2\, dx -&\int_{\Omega(e)}|U_{\theta}^+|^2 |\na \xi_R|^2 \, dx \leq \sum_{i,j=1}^m\int_{\Omega(e)} \frac{\de f_i}{\de u_j}(|x|,U)v_R^j v_R^i \, dx.
\end{align*}
Then
\begin{align*}
Q_U(v_R,\Omega(e))=&\int_{\Omega(e)}\big( |\na v_R|^2-\sum_{i,j=1}^m \frac{\de f_i}{\de u_j}(|x|,U)v_R^j v_R^i \big)\, dx\\
&\leq \int_{\Omega(e)}|U_{\theta}^+|^2 |\na \xi_R|^2 \, dx\leq \frac 1{R^2}\int_{B_{2R}\setminus B_R} |U_{\theta}|^2\, dx\\
&\leq \frac 1{R^2}
\int_{B_{2R}\setminus B_R} |x|^2 |\na U|^2\, dx\leq 4 \int_{B_{2R}\setminus B_R}|\na U|^2\, dx.
\end{align*}
Since $|\na U|\in  L^2 (\Omega)$ 
we conclude that $\limsup_{R\to \infty} Q_U  (v_R , \Omega(e) )\leq 0$, so that from \eqref{nuova} we have
$$\int_{\Omega(e)}\na U_{\theta}^+\cdot \na \phi -\frac 12\sum_{i,j=1}^m \big(\frac{\de f_i}{\de u_j}(|x|,U) +\frac{\de f_j}{\de u_i}(|x|,U)\big) \left( U_{\theta}^j\right)^+\phi^i\, dx=0.$$
Since $\phi \in C_c^{1} (\Omega (e), \R^m)$  was chosen arbitrarily, we conclude that $U_{\theta}^+$ is a solution of the symmetric system  associated with the linearized system, i.e. $U_{\theta}^+$ is a weak solution of
\begin{equation}\nonumber\label{lin-simm}
\left\{
\begin{array}{ll}
-\Delta \left(U_{\theta}^1\right)^+ -\frac 12 \sum_{j=1}^m \big(\frac{\de f_1}{\de u_j}(|x|,U) +\frac{\de f_j}{\de u_1}(|x|,U)\big) \left( U_{\theta}^j\right)^+=0 &\text{ in $\Omega(e )$}\\
\dots\dots \\
-\Delta \left(U_{\theta}^m\right)^+ -\frac 12 \sum_{j=1}^m \big(\frac{\de f_m}{\de u_j}(|x|,U) +\frac{\de f_j}{\de u_m}(|x|,U)\big) \left( U_{\theta}^j\right)^+=0 &\text{ in $\Omega(e )$}\\
U_{\theta}^1=\dots=U_{\theta}^m=0 &\text{ on $\de\Omega(e )$}
\end{array}\right.
\end{equation}
and, as remarked before, this symmetric system is fully coupled in $\Omega(e)$. The Strong Maximum Principle then implies that either $U_{\theta}^+\equiv 0$ or $U_{\theta}^+ >0$ in $\Omega(e)$ and this concludes the proof.
\end{proof}

\noindent Now we need to recall the following definition
\begin{definition}
Let $U$ be a $C^2(\Omega,\R^m)$ solution of \eqref{1} and \eqref{2}. We say that $U$ is stable outside a compact set $\mathcal{K}\subset \Omega$ if $Q_U(\psi,\Omega \setminus\mathcal{K})\geq 0$ for any $\psi\in C^1_c(\Omega\setminus \mathcal{K},\R^m)$.
\end{definition}
\noindent Obviously, we have
\begin{remark}\rm If $U$ has finite Morse index, then $U$ is stable outside a compact set $\mathcal{K}\subset \Omega$.
\end{remark}
\noindent Using the stability outside a compact set, we now derive, by means of a rotating plane
argument, the following proposition which guarantees the foliated Schwarz symmetry of $U$.

\begin{proposition}
\label{rotating-plane}
Let $U$ be a solution of \eqref{1} and \eqref{2} such that $|\na U|\in L^2(\Omega)$ and that $U$ is
stable outside a compact set $\mathcal{K}\subset \Omega$.
Assume that one of the following holds:
\begin{itemize}
\item[$\ast)$] suppose that assumptions $i)$, $ii)$ and $iii)$ of Theorem \ref{tfconvessa} are satisfied;
\item[$\ast\ast)$]   suppose that assumptions $i)$, and $ii)$  of Theorem \ref{tf'convessa} are satisfied.
\end{itemize}
If there exists a direction
$e\in S^{N-1}$ such that
\begin{equation}\label{2.10}
 W^e(x)>0 \quad \text{ or }\quad W^e(x)<0
\quad\quad
\text{ for any }x\in\Omega(e),
\end{equation}
then $U$ is foliated Schwarz symmetric.
\end{proposition}
\begin{proof}
Let us assume that $e=(1,0,\dots,0)$ and that $W^e<0$ in $\Omega(e)$. We consider an arbitrary unitary vector $e'\in S^{N-1}$ different from $\pm e$. 
We want to show that either $U\geq U^{\sigma_{e'}}$ in $\Omega(e')$ or $U\leq U^{\sigma_{e'}}$ in $\Omega(e')$. Then, since the vector $e'$ is arbitrary, the foliated Schwarz symmetry follows from Lemma \ref{lfss}.\\
After an orthogonal change of variable that leaves $e=e_1$ invariant we can assume $e'=e_{\beta}=(\cos \beta,\sin \beta,0,\dots,0)$ for some $\beta\in (0,\pi)$. We set $e_{\beta}=(\cos \beta,\sin\beta,0,\dots,0)$ for $\beta\geq 0$, so that $e=e_0$.
As before we write in short
$$\Omega_{\beta}:= \Omega(e_\beta)=\{x\in \Omega\,:\, x_1\cos \beta +x_2\sin
\beta>0\}\quad \hbox{ and } $$
$$W^{\beta}:=W^{e_{\beta}}\quad, \quad T_{\beta}:=T(e_{\beta}) .$$
Then we start rotating planes and we define
$$\tilde{ \beta}=\sup \{\beta \in [0,\pi)\: :\:
\text{$W^{\beta'}\leq 0$ in  $\Omega_{\beta'}$ for all $\beta'\in
  [0,\beta)$}\}.
$$
Our aim is to show that $W^{\tilde{\beta}}\equiv 0$ in $\Omega_{\tilde{\beta}}$.\\
Indeed in this case we can apply Proposition \ref{proposition-B} getting that the angular derivative $U_{\theta}$ in the cylindrical coordinates $(r,\theta,\tilde x)$ is nonnegative. Then Propositon \ref{proposition-A} and Lemma \ref{lfss} give the foliated Schwarz symmetry of $U$.\\
We observe that, by continuity, $W^{\tilde{\beta}}\leq 0$ and
hence
$\tilde{\beta}<\pi$, because $W^{\pi}=-W^0>0$ in $\Omega_{\pi}=-\Omega_0$. \\
Arguing by contradiction we assume that  $W^{\tilde{\beta}}\not\equiv 0$. The function $W^{\tilde{\beta}}$ satisfies both the linear systems
 \eqref{diff-2} and \eqref{diff-1} in $\Omega_{\tilde{\beta}}$. Moreover
if $\ast)$ or  $\ast\ast)$ are satisfied then these linear systems are fully coupled in $\Omega_{\tilde{\beta}}$ by Lemma \ref{l-differenza}.
So by the strong maximum principle  $W^{\tilde{\beta}}<0$ in
$\Omega_{\tilde{\beta}}$.
Moreover, applying the  Hopf's Lemma   on the hyperplane $T_{\tilde{\beta}}$, where
$W^{\tilde{\beta}}$ vanishes, we have, by  Theorem \ref{SMP}
\begin{equation}
\frac{\de W^{\tilde{\beta}}}
{\de e_{\tilde{\beta}}}(x)<0 \quad \text{ for any }\, x\in T_{\tilde{\beta}}\cap \Omega.
\label{2.12}
\end{equation}
Since, by hypotheses, $U$ is stable outside a compact set, there exists
$R_0>0$ such that
\begin{equation}
  \label{eq:extra-cs-0}
Q_U(\psi,\Omega\setminus B_{R_0})\geq 0 \qquad \text{for every $\psi\in
  C^1_c(\Omega\setminus B_{R_0},\R^m)$.}
\end{equation}
We fix $R_1>R_0$, and we claim that there exists $\varepsilon_1>0$ such that
\begin{equation}\label{2.13}
W^{\tilde{\beta} +\e}(x)\leq 0\quad \hbox{ in }B_{R_1}\cap
\Omega_{\tilde{\beta}+\e} \,\, \forall\, \e \in [0,\e_1).
\end{equation}
In the case $\Omega =\R^N\setminus B$, let $B_{\delta}$ be a neighborhood of $\de B$ in $\Omega$ of small
measure to allow the maximum principle to hold in $B_{\delta}$
for the operator $-\Delta+\widetilde B^{{\tilde{\beta}+\e}}(x)$  in case $\ast)$ is satisfied or for the operator $-\Delta+ B^{{\tilde{\beta}+\e}}(x)$ in case $\ast\ast)$ holds, for sufficiently small
$\e>0$, see Theorem \ref{p-di-max-domini-piccoli}. We first show that
 \begin{equation}\label{2.13a}
W^{\tilde{\beta} +\e}(x)\leq 0\quad \hbox{ in }B_{R_1}\cap\left(
\Omega_{\tilde{\beta}+\e}\setminus B_{\delta}\right) \,\, \forall\, \e\in [0,\e_1).
\end{equation}
If (\ref{2.13a}) is not true, we have sequences $\e_n\to
0$ and  $x_n \in B_{R_1}\cap
\left(\Omega_{\tilde{\beta}+\e_n}\setminus B_{\delta}\right)$ such that
$W^{\tilde{\beta}+\e_n}(x_n)
>0$. After passing to a subsequence,
$x_n\to x_0\in \overline{B_{R_1}\cap \left(
  \Omega_{\tilde{\beta}}\setminus B_{\delta}\right)}$ and
$W^{\tilde{\beta}}(x_0)=0$, hence  $x_0\in T_{\tilde{\beta}}$.
Since $W^{\tilde{\beta}+\e_n}(x)=0$ on
$T_{\tilde{\beta}+\e_n}$ and \mbox{$W^{\tilde{\beta}+\e_n}(x_n)>0$}, there should be points
$\xi_n$ on the line segment joining $x_n$ with $T_{\tilde{\beta}+\e_n}$ and
perpendicular to $T_{\tilde{\beta}+\e_n}$,  such that
$\frac{\de W^{\tilde{\beta}+\e_n}}{\de e_{\tilde{\beta}+\e _n}}(\xi_n)>0$.
Passing to the limit we get $\frac{\de W^{\tilde{\beta}}}
{\de e_{\tilde{\beta}}}(x_0)\geq 0$ in contradiction with
(\ref{2.12}). So we get (\ref{2.13a}).\\
By the  maximum principle, the definition of $B_{\delta}$ and
(\ref{2.13a}) we get $W^{\tilde{\beta}+\e}\leq 0$, under both assumptions $\ast)$ and $\ast\ast)$, also in $B_{R_1}\cap
\Omega_{\tilde{\beta}+\e}\cap  B_{\delta}$ and hence (\ref{2.13})
holds.\\
If $\Omega=\R^N$ by the same argument, directly from (\ref{2.12}) we get (\ref{2.13}).\\
Now we want to prove that
\begin{equation}\label{2.14}
W^{\tilde{\beta}+\e} \leq 0 \quad \hbox{ in }\Omega_{\tilde{\beta}+\e} \quad
\text{for all $\e\in [0,\e_1)$.}
\end{equation}
Because $R_0<R_1$, by (\ref{2.13}) the function
$v:=(W^{\tilde{\beta} +\varepsilon})^+\: \chi_{\stackrel{}{\Omega_{\tilde{\beta}+\e}}}$ has its support strictly contained
in $\Omega_{\tilde{\beta}+\e}\setminus B_{R_0}$. We claim that
\begin{equation}
  \label{eq:equiv0}
v \equiv 0.
\end{equation}
We first consider the case where assumption $\ast)$ is satisfied.
Let $\phi \in
C_c^\infty(\Omega_{\tilde{\beta}+\e}\setminus B_{R_0},\R^m)$. By \eqref{eq:extra-cs-0}, the bilinear form
$P_U$, defined in \eqref{prod-scalare} defines a (semidefinite) scalar product on
$C^1_c(\Omega\setminus B_{R_0},\R^m)$ and also on $C^1_c(\Omega_{\tilde{\beta}+\e}\setminus B_{R_0},\R^m)$, and the corresponding
Cauchy-Schwarz-inequality yields
$$
\left(P_U(\psi,\varphi,\Omega_{\tilde{\beta}+\e}\setminus B_{R_0})\right)^2 \le P_U(\psi,\psi,\Omega_{\tilde{\beta}+\e}\setminus B_{R_0}) P_U(\varphi,\varphi,\Omega_{\tilde{\beta}+\e}\setminus B_{R_0})$$
for
  all 
$\psi, \varphi$ in $H^1_0(\Omega_{\tilde{\beta}+\e}
\setminus B_{R_0},\R^m)$ that  vanish a.e. outside a bounded
set. Consequently, we obtain
\begin{equation}
  \label{eq:cs-new-new}
\left(P_U(v_R,\phi, \Omega_{\tilde{\beta}+\e}\setminus B_{R_0})\right)^2 \le Q_U(v_R, \Omega_{\tilde{\beta}+\e}\setminus B_{R_0})\: Q_U (\phi,\Omega_{\tilde{\beta}+\e}\setminus B_{R_0}) \qquad
\text{for $R>0$},
\end{equation}
where $v_R=v\xi_R$ and $\xi_R$ is defined in \eqref{eq:def-xi_R}. The function $v$ is nonnegative in  $ \Omega_{\tilde{\beta}+\e}$, so we are in position to apply Lemma~\ref{l2.2}, part $i)$, getting from  \eqref{Q_u<0}
$$
\limsup_{R \to \infty} Q_U(v_R, \Omega_{\tilde{\beta}+\e}\setminus B_{R_0}) \le 0.
$$
Combining this with \eqref{eq:cs-new-new}, we find that
\begin{align*}
\int_{\Omega_{\tilde{\beta}+\e}\setminus B_{R_0}}\sum_{i=1}^m \nabla v_i \cdot\nabla \phi_i -\frac 12 \sum_{i,j=1}^m\left( \frac{\de f_i}{\de u_j}(|x|,U)+\frac{\de f_j}{\de u_i}(|x|,U)\right)  v_j
\phi_i\,\:dx \\
= \lim_{R \to \infty}P_U(v_R,\phi,\Omega\setminus B_{R_0})= 0.
\end{align*}
Since $\phi \in C_c^\infty(\Omega_{\tilde{\beta}+\e}\setminus B_{R_0},\R^m)$ was chosen
arbitrarily, we conclude that $v$ is a weak solution of the linear symmetric system
$$
-\Delta v_i -\frac 12 \sum_{j=1}^m \left( \frac{\de f_i}{\de u_j}(|x|,U)+\frac{\de f_j}{\de u_i}(|x|,U)\right)  v_j = 0 \qquad \text{in $\,\,\Omega_{\tilde{\beta}+\e}\setminus B_{R_0}$}
$$
for $i=1,\dots,m$. Then however $v=(W^{\tilde{\beta} +\varepsilon})^+\chi_{ \Omega_{\tilde{\beta} +\varepsilon}} \equiv 0$ by the unique continuation principle, since $(W^{\tilde{\beta}+\varepsilon})
^+ \equiv 0$ in $B_{R_1} \cap \Omega
_{\tilde{\beta}+\e}$ by (\ref{2.13}) and $R_1>R_0$. Hence \eqref{eq:equiv0} holds.\\[.5cm]
Next we consider the case where hypothesis $\ast\ast)$ holds.
Since every function $\tau \in
C^1_c(\Omega_{\tilde{\beta}+\e}\setminus B_{R_0},\R^m)$ can be extended to an
odd function $\tilde{\tau} \in C^1_c(\Omega\setminus B_{R_0},\R^m)$ with
respect to the
reflection at $T_{\tilde{\beta} +\varepsilon}$, we have by \eqref{eq:extra-cs-0}:
$$
Q_{e_{\tilde{\beta}+\varepsilon}s}
(\tau,\Omega_{\tilde{\beta}+\e}\setminus B_{R_0} )=\frac 12 Q_U(\tilde{\tau},\Omega\setminus B_{R_0} )\geq 0$$
for all $\tau \in C^1_c(\Omega_{\tilde{\beta}+\e}\setminus B_{R_0},\R^m)$.
Hence the bilinear form $P_{e_{\tilde{\beta}+\e}s}$ associated with $Q_{e_{\tilde{\beta}+\varepsilon}s}$ defines a (semidefinite) scalar product on
$C^1_c(\Omega_{\tilde{\beta}+\e} \setminus B_{R_0},\R^m)$, and the corresponding
Cauchy-Schwarz-inequality reads
\begin{align}
  \label{eq:chauchy-schwarz-2}
\Big(P_{e_{\tilde{\beta}+\e}s}(&v_R,\phi,\Omega_{\tilde{\beta}+\e}\setminus B_{R_0})\Big)^2 \le  Q_{e_{\tilde{\beta}+\e}s}(\phi,\Omega_{\tilde{\beta}+\e}\setminus B_{R_0})
Q_{e_{\tilde{\beta}+\e}s}(v_R,\Omega_{\tilde{\beta}+\e}\setminus B_{R_0})\nonumber\\
&\leq Q_{e_{\tilde{\beta}+\e}s}(\phi,\Omega_{\tilde{\beta}+\e}\setminus B_{R_0})\limsup_{R\to +\infty}Q_{e_{\tilde{\beta}+\e}s}(v_R,\Omega_{\tilde{\beta}+\e}\setminus B_{R_0})
\end{align}
for $v_R:=v\xi_R$ and for any $\phi \in
C_c^1(\Omega_{\tilde{\beta}+\e}\setminus B_{R_0},\R^m)$.
Using $ii)$ of Lemma \ref{l2.2} with $D=\Omega_{\tilde{\beta}+\e}\setminus B_{R_0}$ then we have
$$\limsup_{R\to +\infty}Q_{e_{\tilde{\beta}+\e}s}(v_R,\Omega_{\tilde{\beta}+\e}\setminus B_{R_0})\leq 0.$$
Then, from  \eqref{eq:chauchy-schwarz-2} it follows that
$$\int_{\Omega_{\tilde{\beta}+\e}\setminus B_{R_0}}\sum_{i=1}^m \na v_i\cdot \na \phi_i+\frac 12\sum_{i,j=1}^m (b_{ij}^{e_{\tilde{\beta}+\e}s}(x)+b_{ji}^{e_{\tilde{\beta}+\e}s}(x))v_j\phi_i\, dx=0$$
for any $\phi \in
C_c^1(\Omega_{\tilde{\beta}+\e}\setminus B_{R_0},\R^m)$. Then $\left(W^{\tilde{\beta}+\e}\right)^+$ is a solution of the linear symmetric system
$$-\Delta w_i^++\frac 12\sum_{j=1}^m \left(b_{ij}^{e_{\tilde{\beta}+\e}s}(x)+ b_{ji}^{e_{\tilde{\beta}+\e}s}(x)\right)  w_j^+=0\,,\, \text{ in } \Omega_{\tilde{\beta}+\e}\setminus B_{R_0}$$
for $i=1,\dots,m$.  Then, the unique continuation principle, implies that
$v\equiv 0$ in $\Omega_{\tilde{\beta}+\e}\setminus B_{R_0}$ and \eqref{eq:equiv0} holds also in this case.\\[.5cm]
As a consequence of \eqref{eq:equiv0}, we have got (\ref{2.14}). Then the definition of $\tilde{\beta}$ implies that
  $W^{\tilde{\beta}}\equiv 0$ in $\Omega_{\tilde \beta}$ and this gives the claim.

\end{proof}

\section{\textbf{Proofs of Theorem \ref{tfconvessa} and Theorem \ref{tf'convessa}}}\label{se:4}
\begin{proposition}\label{lf}
Let $U$ be a solution of \eqref{1} and \eqref{2} with Morse index $\mu(U)\leq N$ and assume that the system is fully coupled along $U$ in $\Omega(e)$ for any $e\in S^{N-1}$.
Then there exists a direction $e\in S^{N-1}$ such that
\begin{equation}\label{lll}
Q_U(\psi,\Omega(e))\geq 0\quad \hbox{ for any }\psi \in C^1_c(\Omega(e), \R^m).
\end{equation}
\end{proposition}
\begin{proof}
The case $\mu(U)< 2$ is immediate. Indeed
at least one among $Q_U(-,\Omega(e))$ and $Q_U(-,\Omega(-e))$ should be positive semidefinite, otherwise we would obtain a $2$-dimensional subspace of $ C^1_c(\Omega, \R^m)$ where the quadratic form $Q_U(-,\Omega)$ is negative definite, contradicting the assumption of Morse index less than $2$.\\
So let us assume $2\leq j:=\mu(U)\leq N$. By definition, $j$ is the maximal dimension of a subspace $X_j:=span\{\Psi_1,\dots,\Psi_j\}\subset C^1_c(\Omega,\R^m)$ such that $Q_U(\psi,\Omega)<0$ for any $\psi \in X_j\setminus \{0\}$. We take a ball $B_{\rho}$ with radius $\rho>0$ sufficiently large to contain the supports of all $\Psi_i$, $i=1,\dots,j$. For $R\geq \rho$, in the domain $\Omega\cap B_{R}$ the linearized operator $L_U$ defined in \eqref{linearized} has exactly $j$ negative symmetric eigenvalues, $\l_1^s(L_U,\Omega\cap B_{R})<\l_2^s(L_U,\Omega\cap B_{R})\leq\dots\leq\l_j^s(L_U,\Omega\cap B_{R})$  with respect to Dirichlet boundary conditions, and
$\l_{j+1}^s(L_U,\Omega\cap B_{R})\geq 0$. See Section \ref{se:2} for the definition of the symmetric eigenvalues and \cite{DAPA} for their variational characterization.\\
Now assume, arguing by contradiction, that  for any $e\in S^{N-1}$ \eqref{lll} does not hold. Then, we can apply Lemma \ref{ln2} to the linear operator $L_U$ and
we can find a $\tilde R>0$ such that, for any $R\geq \tilde R$ and  for any $e\in S^{N-1}$ the first symmetric eigenvalue
$\l_1^s(L_U, \Omega(e)\cap B_R)$ of the linearized operator in $\Omega(e)\cap B_R$, with
zero Dirichlet boundary condition is negative. \\
We can take $R\geq \max\{\rho, \tilde R,\bar R\}$ where $\bar R$ is as in Lemma \ref{ln1}.
In this way we have that the linearized system, defined in \eqref{linearized}, is fully coupled in $\Omega(e) \cap B_R$ for any $e\in S^{N-1}$  and the same holds for the symmetric system 
associated with the linearized operator in $\Omega(e) \cap B_R$.
We denote by $\Phi_e$ the first positive $L^2$-normalized eigenfunction of the symmetric system $-\Delta -\frac 12 \left(J_F(|x|,U(x))+J_F^t(|x|,U(x)\right)$ in $\Omega(e)\cap B_R$ (we observe that $\Phi_e$ is uniquely determined since the system is fully coupled in $\Omega(e)\cap B_R$ for any $e\in S^{N-1}$) and by $\Phi_1,\dots,\Phi_j$ the mutually orthogonal eigenfunctions corresponding to the $j$ negative symmetric eigenvalues of $L_U$ in $\Omega \cap B_R$.
Define
$$\Psi_e(x)=\left\{\begin{array}{ll}
 \left(\frac{(\Phi_{-e},\Phi_1)_{L^2(\Omega(-e))}}{(\Phi_{e},\Phi_1)_{L^2(\Omega(e))}}\right)^{\frac 12}\Phi_e(x) & \hbox{ if }x\in \Omega(e)\cap B_R\\
 -\left(\frac{(\Phi_{-e},\Phi_1)_{L^2(\Omega(-e))}}{(\Phi_{e},\Phi_1)_{L^2(\Omega(e))}}\right)^{\frac 12}\Phi_{-e}(x) & \hbox{ if }x\in \Omega(-e)\cap B_R
\end{array}
\right.$$
where $(-,-)_{L^2(D)}$ denotes the scalar product in $L^2(D,\R^m)$.
The mapping $e\mapsto \Psi_e$  is a continuous odd function from $S^{N-1}$ to $H^1_0(\Omega, \R^m)$ and by construction $(\Psi_e,\Phi_1)_{L^2(\Omega\cap B_R)}=0$. Therefore the mapping
$h:S^{N-1}\to \R^{j-1}$ defined by
$$h(e)=\left( (\Psi_e, \Phi_2)_{L^2(\Omega \cap B_R)},\dots,(\Psi_e, \Phi_j)_{L^2(\Omega \cap B_R)}\right)$$
is an odd continuous mapping, and since
$j-1<N$ by the Borsuk-
Ulam Theorem it must have a zero. This means that there exists a
direction $e\in S^{N-1}$ such that $\Psi_e$ is orthogonal to all the eigenfunctions
$\Phi_1,\dots,\Phi_j$ in $L^2(\Omega \cap B_R, \R^m)$. Since
$\mu(U)=j$ this implies that $\frac{Q_U(\Psi_e, \Omega\cap B_R)}{( \Psi_e,\Psi_e) _ {L^2(\Omega \cap B_R)}   }\geq \l_{j+1}^s(L_U, \Omega\cap B_R)\geq 0$ against the fact that
\begin{align*}
Q_U(\Psi_e, \Omega\cap B_R)=& \left(\frac{(\Phi_{-e},\Phi_1)_{L^2(\Omega(-e))}}{(\Phi_{e},\Phi_1)_{L^2(\Omega(e))}}\right) \l_1^s\left(L_U, \Omega(e)\cap B_R\right)\\
+&\left(\frac{(\Phi_{-e},\Phi_1)_{L^2(\Omega(-e))}}{(\Phi_{e},\Phi_1)_{L^2(\Omega(e))}}\right)\l_1^s\left( L_U,\Omega(-e)\cap B_R \right)<0
\end{align*}
by construction. The contradiction proves the assertion.
\end{proof}

\noindent Now we are in position to prove Theorem \ref{tfconvessa}.
\begin{proof}[\bf{Proof of Theorem \ref{tfconvessa}}]
By Proposition \ref{lf} we have that there exists a direction $e \in  S^{N-1}$ such that \eqref{lll} holds. Hypothesis i) of Theorem \ref{tfconvessa} implies that the system \eqref{1} is fully coupled along $U$ also in $\Omega$ and so, 
if $W^e \equiv 0$  in $\Omega (e)$, we immediately get the foliated Schwarz symmetry of $U$, by Proposition \ref{pf}.\\
If instead $W^e\not\equiv 0$ we show that $W^e$ is eather strictly positive or strictly negative in $\Omega(e)$ so that, by Proposition \ref{rotating-plane} we again get the assertion.\\
From \eqref{lll} we have that  $P_U(\psi,\phi,\Omega(e))$ is a semidefinite scalar product on $C^1_c(\Omega(e),\R^m)$.
Consequently, using the  Cauchy-Schwarz-inequality
we obtain:
\begin{equation}\label{f5}
\left(
P_U (v_R, \phi, \Omega(e))\right)^2\leq  Q_U (v_R,\Omega(e) )Q_U (\phi,\Omega(e) )
\end{equation}
for any $ \phi \in  C^1_c (\Omega(e), \R^m )$,
where $v_R:=(W^e)^+\chi_{\Omega(e)}\xi_R$ and $\xi_R$ is a cut-off function as defined in \eqref{eq:def-xi_R}. By Lemma \ref{l2.2} we have $\limsup_{R\to +\infty}Q_U (v_R, \Omega(e))\leq 0$. Combining this with \eqref{f5} and passing to the limit as $R\to +\infty$ we have
\begin{align*}
\int_{\Omega(e)}\na (W^e)^+ \cdot \na \phi-\frac 12\sum_{i,j=1}^m \bigg(\frac{\de f_i}{\de u_j}(|x|,&U)+\frac{\de f_j}{\de u_i}(|x|,U)\bigg)w_j^+ \phi_i \, dx\\
&=\lim_{R\to +\infty}P_U (v_R, \phi, \Omega(e) )=0.
\end{align*}
Since $\phi \in C_c^{1} (\Omega (e), \R^m)$  was chosen arbitrarily, we conclude that $(W^e)^+$ is a solution of the system
\begin{equation}\label{eq-inf}
-\Delta w_i^+-\frac 12 \sum_{j=1}^m \bigg(\frac{\de f_i}{\de u_j}(|x|,U)+\frac{\de f_j}{\de u_i}(|x|,U)\bigg)w_j^+=0\quad \text{ in }\Omega(e )
\end{equation}
for $i=1,\dots, m$.
Now, since  $(W^e)^+\geq 0$ in $\Omega(e )$ and the linear system \eqref{eq-inf} is fully coupled in $\Omega(e)$, the Strong Maximum Principle implies that either $(W^e)^+\equiv 0$ or $ (W^e) ^+ >0$ in $\Omega(e)$. In any case $W^e$ is strictly positive or strictly negative in $\Omega(e)$. 
\end{proof}
\begin{proposition}\label{p8}
Let $U$ be a solution of \eqref{1} and \eqref{2} with Morse index $\mu(U)=j\leq N-1$ and assume that the system \eqref{1} is fully coupled along $U$ in $\Omega$. Then there exists a direction $e\in S^{N-1}$ such that
\begin{equation}\label{8}
Q_{es}(\psi,\Omega(e)):=\int_{\Omega(e)}|\na \psi|^2 +\sum_{i,j=1}^m b_{ij}^{es}(x)\psi^i\psi^j\, dx \geq 0
\end{equation}
for any $\psi \in C^1_c(\Omega(e),\R^m)$.
\end{proposition}
\begin{proof}
Assume, arguing by contradiction, that for any $e\in S^{N-1}$ \eqref{8} does not hold. Then we can apply Lemma \ref{ln2} to the linear operator $L^{es}$, defined in \eqref{equation-b-es}, and we can find $\tilde R>0$ such that, for any $R\geq \tilde R$ and for any $e\in S^{N-1}$the first symmetric eigenvalue $\l_1^s(L^{es}, \Omega(e)\cap B_R)$ of the linear operator $L^{es}$ in $\Omega(e)\cap B_R$ with Dirichlet boundary conditions is negative. \\
By definition $j$ is the maximal dimension of a subspace  $X_j:=span<\Psi_1,\dots,\Psi_j>\subset C^1_c(\Omega, \R^m)$ such that $Q_U(\psi,\Omega)<0$ for any  $\psi \in X_j\setminus\{0\}$.
We take a ball $B_{\rho}$ with radius $\rho>0$ sufficiently large to contain the supports of all $\Psi_i$, $i=1,\dots,j$. For $R\geq \rho$, in the domain $\Omega\cap B_R$ the linearized operator $L_U$ has exactly $j$ negative symmetric eigenvalues  and $\l_{j+1}^s(L_U, B_{R}\cap \Omega)\geq 0$.\\
We take  $R\geq\max\{\rho,\tilde R, \bar R\}$, where $\bar R$ is as defined in Remark \ref{lfa}.
In this way we have that the linear system $-\Delta +B^{es}(x)$ defined in \eqref{equation-b-es}, is fully coupled in $\Omega(e)\cap B_R$ for any $e\in S^{N-1}$ and the same holds for the symmetric system associated with the linear operator $L^{es}$ in $\Omega(e)\cap B_R$. Moreover the system \eqref{1} is fully coupled along $U$ in $\Omega\cap B_R$.\\
We denote by $g_e$ the first positive $L^2$-normalized eigenfunction corresponding to $\l_1^s(L^{es}, \Omega(e)\cap B_R)$ (which is defined since the system $-\Delta V+B^{es}(x)V$ is fully coupled in  $\Omega(e)\cap B_R$ for any $e\in S^{N-1}$) and by $\tilde g_e$ the odd extension to $\Omega \cap B_R$. Let $\Phi_1,\dots,\Phi_j$ be the mutually orthogonal eigenfunctions corresponding to the $j$ negative symmetric eigenvalues of $L_U$ in $\Omega \cap B_R$.
Let $h: S^{N-1}\ra R^j$ be defined by
$$h(e):=\left((\tilde{g}_e,\Phi_1)_{L^2(\Omega \cap B_R)},\dots,(\tilde{g}_e,\Phi_j)_{L^2(\Omega \cap B_R)}\right)$$
where $(-,-)_{L^2(D)}$ is, as before, the usual scalar product in $L^2(D,\R^m)$.
$h$ is an odd and continuous mapping and, since $j<N-1$ it must have a zero by the Borsuk-Ulam Theorem. This means that there exists a direction $e\in S^{N-1}$ such that $\tilde g_e$ is orthogonal to all the eigenfunctions $\Phi_1, \dots, \Phi_j$ in  $L^2(\Omega\cap B_R,\R^m)$. Since $\mu(U)=j$ this implies that
$$\frac{Q_U(\tilde g_e, \Omega\cap B_R)}{(\tilde g_e,\tilde g_e)_{L^2(\Omega \cap B_R)}}=\frac{\int_{\Om\cap B_R}|\na \tilde g_e|^2 -\sum_{i,j=1}^m \frac{\de f_i}{\de u_j}(|x|,U)\tilde g_e^i\tilde g_e^j\, dx} {(\tilde g_e,\tilde g_e)_{L^2(\Omega \cap B_R)}}       \geq \l_{j+1}^s(L_U, \Omega\cap B_R)\geq 0.$$
The symmetry properties of $\tilde g_e$ imply that
\begin{align*}
\int_{\Om}|\na \tilde g_e|^2 &-\sum_{i,j=1}^m \frac{\de f_i}{\de u_j}(|x|,U)\tilde g_e^i\tilde g_e^j\, dx=2\int_{\Om(e)\cap B_R}|\na g_e|^2 +\sum_{i,j=1}^m b_{ij}^{es}(x)g_e^i g_e^j\, dx\\
&=2\l_1^{s}( L^{es}, \Omega(e) \cap B_R)<0.
\end{align*}
The contradiction proves the assertion.
\end{proof}

\begin{proof}[\bf{Proof of Theorem \ref{tf'convessa}}]
From Proposition \ref{p8} we have a direction $e \in S^{N-1}$ such that $Q_{es}(\psi, \Omega(e))\geq 0$ for any $\psi \in C_c^1(\Omega(e), \R^m)$. 
If $W^e\equiv 0$, i.e. if $U$ is symmetric with respect to $T(e)$, then  $Q_{es}(\psi,\Omega(e))=Q_U(\psi,\Omega(e))\geq 0$ for any $\psi \in C^1_c(\Omega(e),\R^m)$. The foliated Schwarz symmetry of $U$ then follows from  Proposition \ref{pf}.\\
So assume $W^e\not\equiv 0$.
By assumption the bilinear form $P_{es}$, defined in \eqref{P-es}, defines 
a scalar product on $ C^1_c(\Omega(e),\R^m)$, so, 
for $v_R=(W^e)^+\xi_R\in H^1_0(\Omega(e),\R^m)$
and for any 
$\phi \in C^1_c(\Omega(e),\R^m)$, we have
\begin{equation}\label{**}
0\leq \Big(P_{es}(v_R,\phi,\Omega(e))\Big)^2 \leq Q_{es}(v_R,\Omega(e))Q_{es}(\phi,\Omega(e)).
\end{equation}
Moreover, using \eqref{Q_es<0} with $D=\Omega(e)$, we have
$$\limsup_{R\to +\infty}Q_{es}(v_R,\Omega(e))\leq 0.$$
Passing to the limit in \eqref{**} we get
$$P_{es}(\big(W^e\big)^+,\phi,\Omega(e))=0$$
for any $\phi \in C^1_c(\Omega(e),\R^m)$, so that $\big(W^e\big)^+$ is a weak solution of
\begin{equation}\label{sist-inf}
-\Delta(w_i)^+ +\frac 12 \sum_{j=1}^m (b^{es}_{ij}(x)+b^{es}_{ji}(x))(w_j)^+=0\quad \text{ in } \Omega(e)
\end{equation}
for $i=1,\dots,m$.
Since the system is fully coupled in $\Omega(e)$, the strong maximum principle implies that either $\big(W^e\big)^+>0$ in $\Omega(e)$ or $\big(W^e\big)^+\equiv 0$ in $\Omega(e)$. In any case $W^e$ is strictly positive or strictly negative  in $\Omega(e)$
and the foliated Schwarz symmetry  of $U$ follows from Proposition \ref{rotating-plane}.
\end{proof}
\section{\textbf{Other results}}
\label{se:5}
\noindent We prove the other theorems stated in Section \ref{se:1}.
\begin{proof}[Proof of Theorem \ref{corollario1}]
If  $U$ is a Morse index one solution for any direction $e\in S^{N-1}$  at least one among $Q_U(-,\Omega(e))$ and $Q_U(-,\Omega(-e))$ should be positive semidefinite, otherwise we would obtain a $2$-dimensional subspace of $ C^1_c(\Omega, \R^m)$ where the quadratic form $Q_U(-,\Omega)$ is negative defined, contradicting the definition of Morse index $1$.\\
By the  proof of Theorem \ref{tfconvessa} and of Theorem \ref{tf'convessa} we can find a direction $e\in S^{N-1}$ such that $W^e\equiv 0$ in $\Omega(e)$ or $W^e>0$ in $\Omega(e)$. In the second case, applying Proposition \ref{rotating-plane} we can find a direction $e'$ such that $W^{e'}\equiv 0$ in $\Omega(e')$. So, in any case, there exists a direction $e$ such that $U$ is symmetric with respect to  the hyperplane $T(e)$.
Thus, by symmetry, $Q_U(\psi,\Omega(e))=Q_U(\psi,\Omega(-e))$ for any $\psi\in C^1_c(\Omega(e),\R^m)$ and $Q_U$ is positive semidefinite in $\Omega(e)$.\\
After a rotation, we may assume that $e=e_2=(0,1,\dots,0)$ so that $T(e)\,=\,\{x\in \R^N\,:\,x_2=0\}$
and we may introduce new (cylinder) coordinates
$(r, \theta, y_3, \dots ,y_N)$ defined by the relations $x= r [\cos \theta e_1 + \sin \theta e_2] + \sum_{i=3}^N y_i e_i $.   \\
Then the angular derivative $U_{\theta} $ of $U$ with respect to $\theta $, extended by zero at  the origin if $\Omega $ is a ball, satisfies the linearized system, i.e.
\begin{equation}\label{u-theta-+}
- \Delta U _{\theta} - J_F (|x|, U) U_{\theta} =0 \quad \text{ in } \Omega(e_2)
\end{equation}
and it is zero on the boundary.
Reasoning exactly as in the proof of Proposition \ref{pf} we have that $U_{\theta}^+$ is a solution of
\begin{equation}\label{simm-u-theta-+}
-\Delta U_{\theta}^+- \frac 12 \left(J_F (|x|, U) +J_F ^t(|x|, U)\right) U_{\theta}^+ =0 \quad \text{ in } \Omega(e_2)
\end{equation}
and this implies that $U_{\theta}$ does not change sign in $\Omega(e_2)$. We can assume that $U_{\theta}=U_{\theta}^+$, and that $U_{\theta}^+$ is a solution of the systems  \eqref{u-theta-+} and \eqref{simm-u-theta-+}. Then
$J_F (|x|, U) U_{\theta}= \frac 12 \left (J_F (|x|, U) + J_F ^t (|x|, U) \right ) U_{\theta}$, i.e.
  \eqref{superfullycoupling1} and if $m=2$, since $U_{\theta}$ is positive, we get    \eqref{superfullycoupling2}. The case $U_{\theta}=U_{\theta}^-$ can be handled in the same way. This proves the assertion if the hypothesis $a)$ holds.\\
To prove the theorem under assumption $b)$, we observe that the result follows if we can find a direction $e$ such that $W^e\equiv 0$ in $\Omega(e)$ and $Q_U(\psi,\Omega(e))\geq 0$ for any $\psi\in C^1_c(\Omega(e),\R^m)$.\\
Following the proof of Theorem \ref{tf'convessa} we have a direction $e\in S^{N-1}$ such that either $W^e\equiv 0$ and $Q_U(\psi,\Omega(e))\geq 0$ or $W^e>0$ in $\Omega(e)$.\\
The second case cannot happen. Indeed
the function $W^e$ satisfies the system
\begin{equation}\label{sist-diff}
-\Delta w_i+ \sum_{j=1}^m b_{ij}^e(x)w_j=0\quad \text{ in }\Omega(e )
\end{equation}
where $b_{ij}^e(x)$ are as in Lemma \ref{l-differenza}. Multiplying the $i$-th equation of \eqref{sist-inf} and \eqref{sist-diff} for $w_i\xi_R$ ($\xi_R$ is the usual cutoff function, see \eqref{eq:def-xi_R}), integrating in  $\Omega(e)$, summing on $i$ and subtracting, we get
$$\sum_{i,j=1}^m\int_{\Omega(e)}\left( b_{ij}^e(x)-\frac 12( b^{es}_{ij}(x)+b^{es}_{ji}(x))\right) w_jw_i\xi_R\, dx=0$$
equivalently
\begin{equation}\label{fin2}
\sum_{i,j=1}^m\int_{\Omega(e)}\left( b_{ij}^e(x)-b_{ij}^{es}(x)\right) w_jw_i\xi_R\, dx=0.
\end{equation}
Since $W^e>0$, Lemma \ref{l-differenza} implies that $b_{ij}^e(x)\geq b_{ij}^{es}(x)$ for any $i,j=1,\dots,m$, and since $w_iw_j\xi_R>0$ in $\Omega(e)$, relation \eqref{fin2} gives
$b_{ij}^e(x)= b_{ij}^{es}(x)$ for any $i,j=1,\dots,m$.\\
By the strict convexity of $\frac{\de f_{i_0}}{\de u_{j_0}}$ we have that, if $W^e>0$ then $b_{i_0,j_0}^e(x)>b_{i_0,j_0}^{es}(x) $ getting a contradiction. Therefore the only possible case is $W^e\equiv 0$ in $\Omega(e)$ and $Q_U(\psi,\Omega(e))\geq 0$. Then the result follows as in the previous case.
\end{proof}
\begin{proof}[Proof of Theorem \ref{soluzioni-stabili}]
We choose, as before, the cylindrical coordinates with respect to the plane $x_1x_2$, i.e. $(r,\theta,\tilde x)$. 
Again the derivative 
$U_{\theta}$  satisfies the linearized system 
$-\Delta U_{\theta}-J_F(|x|,U)U_{\theta}=0$  in $\Omega$
and, in the case $\Omega=\R^N\setminus B_R(0)$, also the boundary conditions
$U_{\theta}=0$ on $\de \Omega.$
By the stability  assumption, we can proceed as in the proof of Proposition \ref{pf}, with $\Omega$ in place of $\Omega(e)$, to show that $U_{\theta}$ does not change sign in $\Omega$.
Since $U_{\theta}$ is $2\pi$-periodic this is impossible and therefore $U_{\theta}\equiv 0$.
By the arbitrarily of $x_1,x_2$ we conclude that $U$ is radial.\\
Moreover, if $\Omega=\R^N$ and $F$ does not depend on $|x|$, then for every $t$ the translated function $U(x+t)$ is also a stable solution of \eqref{1}, therefore it is radial by the argument above. This however is not possible unless $U$ is constant.
\end{proof}
\begin{proof}[Proof of Theorem \ref{tnonesistenza1}]
Suppose by contradiction that \eqref{1} admits a sign
changing solution
$U$ on $\R^N$ that satisfies the assumptions of Theorem \ref{corollario1} and such that $\lim \limits_{|x| \to \infty}U(x)=0$. Since we can apply
Theorems \ref{tfconvessa} and \ref{tf'convessa} then $U$ is
foliated Schwarz symmetric. By a rotation of coordinates, we
may assume that $p=e_N$ in the definition of foliated Schwarz
symmetry, so that $U$ is axially symmetric with respect to the axis
$\R e_N$ and nonincreasing in the angle $\theta= \arccos
\frac{x_N}{|x|}$. By the proofs of Theorem \ref{corollario1} we get a direction
$e\in S^{N-1}$ such that $U$ is symmetric with respect to
$T(e)$ and
\begin{equation}\label{4.1}
\inf _{\psi\in C^1_c(\Omega(e))}Q_U(\psi,\Omega(e))\geq 0.
\end{equation}
We may assume that $e=e_1$ in \eqref{4.1}. Indeed, this is clearly
possible if $U$ is radial. Moreover, if $U$ is nonradial, then $U$ is
strictly decreasing in the angle $\theta$, therefore the symmetry
hyperplanes of $U$ are precisely the ones containing $e_N$, and for
each one of them the infimum in \eqref{4.1} takes the
same value.\\
We now consider the derivative $\frac{\de U}{\de x_1}$ which, by regularity theory, (see \cite{GPW} Lemma 6.1)  belongs to $H^2(\Omega(e_1))$ and satisfies the linearized system 
\begin{equation}\label{4.2}
-\Delta \left(\frac{\de U}{\de x_1}\right) -J_F(U(x))\left(\frac{\de U}{\de x_1}\right)=0 \quad\hbox{
  in } \Omega (e_1).
\end{equation}
Moreover, because of the symmetry with respect to $T(e_1)$ we have
$$
\frac{\de U}{\de x_1}=0\quad \hbox{ on } \de \Omega(e_1) 
$$
so that $\frac{\de U}{\de x_1}\in H^1_0(\Omega(e_1))$.
From \eqref{4.1} we have that $P_U$ defines a (semidefinite) scalar product on $C^1_c(\Omega (e_1),\R^m)$ and the corresponding Cauchy-Schwarz-inequality reads
\begin{equation}\label{4.3}
\left(P_U(v_R,\phi,\Omega(e_1))\right)^2\leq Q_U(\phi,\Omega(e_1))\limsup_{R\to +\infty}Q_U(v_R,\Omega(e_1))
\end{equation}
if $v_R=\left(\frac{\de U}{\de x_1}\right)^+ \xi_R$, with $\xi_R$ being the usual cutoff function. 
Reasoning exactly as in the proof of Proposition \ref{pf}, with $\left(\frac{\de U}{\de x_1}\right)^+ $instead of $U_{\theta}^+$, we get 
$$Q_U(v_R,\Omega(e_1))\leq \int_{\Omega(e_1)}\Big| \left(\frac{\de U}{\de x_1}\right)^+\Big|^2|\na \xi_R|^2 \, dx\leq \frac 1{R^2}\int_{B_{2R}\setminus B_R}|\na U|^2\,dx$$
so that
$$\limsup_{R\to +\infty} Q_U(v_R,\Omega(e_1))\leq 0.$$
Then \eqref{4.3} implies, as before, that $\left(\frac{\de U}{\de x_1}\right)^+$ is a weak solution of the system
\begin{equation}\label{4.4}
-\Delta \left(\frac{\de U}{\de x_1}\right)^+-\frac 12 \left(J_F(U(x))+J_F^t(U(x))\right)\left(\frac{\de U}{\de x_1}\right)^+=0\quad \text{ in }\Omega(e_1).
\end{equation}
Since this system is fully coupled in $\Omega(e_1)$ the Strong Maximum principle  implies that either  $\left(\frac{\de U}{\de
  x_1}
\right)^+\equiv 0$ in $\Omega(e_1)$ or $\frac{\de U}{\de x_1}>0$ in
$\Omega(e_1)$. In any case $\frac{\de U}{\de x_1}$ does not change
sign in  $\Omega(e_1)$ which is impossible because $U$ changes sign in
$\Omega(e_1)$ and $\lim \limits_{|x| \to \infty}U(x)=0$.
This contradiction proves the assertion.

\end{proof}
\begin{proof}[Proof of Theorem \ref{tnonesistenza2}]
Suppose by contradiction that \eqref{1} and \eqref{2} admit a positive solution
$U$ in $\R^N\setminus B$ that satisfies  the assumptions of Theorem \ref{corollario1} and such that $\lim \limits_{|x| \to \infty}U(x)=0$.
As in the proof of Theorem \ref{tnonesistenza1}
we may assume that $U$ is symmetric with respect to $T(e_1)$ and that \eqref{4.1} holds for $e=e_1$. Then the derivative $\frac{\de U}{\de x_1}$ satisfies the system
\begin{equation}\label{ult}
\left\{
\begin{array}{ll}
-\Delta \left(\frac{\de U}{\de x_1}\right) -J_F(U(x))\left(\frac{\de U}{\de x_1}\right)=0 &\hbox{
  in } \Omega (e_1)\\
\frac{\de U}{\de x_1}=0 & \hbox{
  on } T (e_1)\cap \overline{\Omega}\\
\frac{\de U}{\de x_1}\geq 0 & \hbox{
  on } \de B\cap \overline{\Omega(e_1)}\\
\end{array}
\right.
\end{equation}
Multiplying  by $\left(\frac{\de U}{\de x_1}\right)^- \xi_R^2$, integrating over $\Omega(e_1)$ and using the cooperativeness of \eqref{1},  we get
$$-\sum_{i=1}^m \int_{\Omega(-e_1)} \na \left(\frac{\de u_i}{\de x_1}\right)\cdot \na \left[\left(\frac{\de u_i}{\de x_1}\right)^- \xi_R^2\right]\, dx\leq \sum_{i,j=1}^m\int_{\Omega(-e_1)}\frac{\de f_i}{\de u_j}(U(x))\left(\frac{\de u_j}{\de x_1}\right)^-\left(\frac {\de u_i}{\de x_1}\right)^-\xi_R^2\, dx.$$
Then, as in the proof of the previous theorem, we get that $(\frac{\de U}{\de x_1})^-$ is a solution of the symmetric system \eqref{ult} and hence $ (\frac{\de U}{\de x_1})^-\equiv 0$ by the boundary conditions. Then $\frac{\de U}{\de x_1}\geq 0$ in $\Omega(e_1)$ contradicting the fact that $U=0$ on $\de B$ and $U(x)\to 0$ as $|x|\to +\infty$.
\end{proof}
\begin{proof}[Proof of Theorem \ref{tnonesistenza3}]
The proof follows as in the case of Theorems \ref{tnonesistenza1} and \ref{tnonesistenza2} 
once we get, as in the proof of Theorem \ref{corollario1}, the existence of a direction $e\in S^{N-1}$ such that $U$ is symmetric with respect to $T(e)$ and that \eqref{4.1} holds.\\
To get \eqref{4.1} we need the following fact: 
if we have a direction $e\in S^{N-1}$ such that $W^e$ is either strictly positive or strictly negative in $\Omega(e)$ and the system is of gradient type, then
\begin{equation}\label{ults}
Q_e(\psi,\Omega(e))\geq 0\quad \text{for any }\psi \in C^1_c(\Omega(e),\R^m).
\end{equation}
This fact is a generalization of Lemma 2.1 in \cite{GPW} and follows in a similar way.\\
Now, starting from the proof of Theorem \ref{tfconvessa} (or of Theorem \ref{tf'convessa}) we get a direction $e\in S^{N-1}$  such that $Q_U(\psi,\Omega(e))\geq 0$  ($Q_{es}(\psi,\Omega(e))\geq 0$ respectively) for any $\psi \in C^1_c(\Omega(e),\R^m)$. If $W^e\equiv 0$ in $\Omega(e)$ then  \eqref{4.1} is satisfied ($Q_{es}(\psi,\Omega(e))=Q_{U}(\psi,\Omega(e))$, by the symmetry) and we are done. If, else, $W^e\not\equiv 0$ we have, as in the proof of  Theorem \ref{tfconvessa} (Theorem \ref{tf'convessa} respectively) that  $W^e$ is either strictly positive or strictly negative in $\Omega(e)$.\\
Then, applying the rotating plane method, see Proposition \ref{rotating-plane}, we get, using the same notations, the existence of $\tilde \beta>0$ such that $W^{\tilde{\beta}}\equiv 0$ in $\Omega(\tilde \beta)$ and $W^{\beta}<0$ in $\Omega(\beta)$ for any $\beta\in [0,\tilde \beta)$. This means, using \eqref{ults}, that $Q_e(\psi,\Omega(\beta))\geq 0$ for any $\beta\in [0,\tilde \beta)$ and, passing to the limit,  $Q_e(\psi,\Omega(\tilde{\beta}))\geq 0$. The symmetry of $U$ with respect to $T(\tilde \beta)$ then implies that
$Q_e(\psi,\Omega(\tilde{\beta}))=Q_U(\psi,\Omega(\tilde{\beta}))\geq 0$ and \eqref{4.1} is satisfied concluding the proof.
\end{proof}

\end{document}